\documentclass{amsart}

\usepackage{amsmath}
\usepackage{amssymb, amsfonts}
\usepackage[all]{xy}
\usepackage{mathrsfs}
\usepackage[titletoc,title]{appendix}
\usepackage[bookmarks]{hyperref}

\newtheorem{thm}{Theorem}[section]
\newtheorem{defn}[thm]{Definition}
\newtheorem{rmk}[thm]{Remark}
\newtheorem{rmks}[thm]{Remarks}
\newtheorem{cor}[thm]{Corollary}
\newtheorem{prop}[thm]{Proposition}
\newtheorem{lem}[thm]{Lemma}

\newtheorem{notation}[thm]{Notations}

\numberwithin{equation}{section}

\newcommand{\id}{{\rm id}}

\begin{document}

\title[Coactions of Hopf $C^*$-algebras on Cuntz-Pimsner algebras]{Coactions of Hopf $C^*$-algebras on Cuntz-Pimsner algebras}

\author{Dong-woon Kim}

\address{Dong-woon Kim: Department of Mathematical Sciences and Research Institute of Mathematics,
Seoul National University, Seoul 151-747, Korea}\email{dwkim0962@gmail.com}

\begin{abstract}
    Unifying two notions of an action and coaction of a locally compact group on a $C^*$-cor\-re\-spond\-ence we introduce a coaction $(\sigma,\delta)$ of a Hopf $C^*$-algebra $S$ on a $C^*$-cor\-re\-spond\-ence $(X,A)$. We show that this coaction naturally induces a coaction $\zeta$ of $S$ on the associated Cuntz-Pimsner algebra $\mathcal{O}_X$ under the weak $\delta$-invariancy for the ideal $J_X$. When the Hopf $C^*$-algebra $S$ is defined by a well-behaved multiplicative unitary, we construct a $C^*$-cor\-re\-spond\-ence $(X\rtimes_\sigma\widehat{S},A\rtimes_\delta\widehat{S})$ from $(\sigma,\delta)$ and show that it has a representation on the reduced crossed product $\mathcal{O}_X\rtimes_\zeta\widehat{S}$ by the induced coaction $\zeta$. This representation is used to prove an isomorphism between the $C^*$-algebra $\mathcal{O}_X\rtimes_\zeta\widehat{S}$ and the Cuntz-Pimsner algebra $\mathcal{O}_{X\rtimes_\sigma\widehat{S}}$ under the covariance assumption which is guaranteed in particular if the ideal $J_{X\rtimes_\sigma\widehat{S}}$ of $A\rtimes_\delta\widehat{S}$ is generated by the canonical image of $J_X$ in $M(A\rtimes_\delta\widehat{S})$ or the left action on $X$ by $A$ is injective. Under this covariance assumption, our results extend the isomorphism result known for actions of amenable groups to arbitrary locally compact groups. The Cuntz-Pimsner covariance condition which was assumed for the same isomorphism result concerning group coactions is shown to be redundant.
\end{abstract}

\maketitle

\section{Introduction}

    In this paper, we introduce coactions of Hopf $C^*$-algebras on $C^*$-cor\-re\-spond\-ences, and study the \emph{induced} coactions on the associated Cuntz-Pimsner algebras and their crossed products.
	
    A $C^*$-cor\-re\-spond\-ence $(X,A)$ is a (right) Hilbert $A$-module $X$ equipped with a left action $\varphi_A:A\rightarrow\mathcal{L}(X)$. For $(X,A)$ with injective $\varphi_A$, a $C^*$-algebra $\mathcal{O}_X$ was constructed in \cite{Pim} generalizing crossed product by $\mathbb{Z}$ and Cuntz-Krieger algebra \cite{CK}. The construction was extended in \cite{Kat} to arbitrary $C^*$-cor\-re\-spond\-ences $(X,A)$ by considering an ideal $J_X$ of $A$ --- the largest ideal that is mapped injectively into $\mathcal{K}(X)$ by $\varphi_A$ --- and requiring that a covariance-like relation should hold on $J_X$. The $C^*$-algebra $\mathcal{O}_X$, called the Cuntz-Pimsner algebra associated to $(X,A)$, is generated by $k_X(X)$ and $k_A(A)$ for the universal covariant representation $(k_X,k_A)$. The class of Cuntz-Pimsner algebras is known to be large enough and include in particular graph $C^*$-algebras. In addition, there have been significant results concerning Cuntz-Pimsner algebras such as gauge invariant uniqueness theorem, criteria on nuclearity or exactness, six-term exact sequence, and description of ideal structure (\cite{Pim,Kat,Kat2}). Thus the Cuntz-Pimsner algebras can be viewed as a well-understood class of $C^*$-algebras, and in view of this, it would be advantageous to know that a given $C^*$-algebra is a Cuntz-Pimsner algebra.

    Our work was inspired by \cite{HaoNg} and \cite{KQRo2} in which group actions and coactions on $C^*$-cor\-re\-spond\-ences are shown to induce actions and coactions on the associated Cuntz-Pimsner algebras, and the crossed products by the induced actions or coactions are proved to be realized as Cuntz-Pimsner algebras. (We refer to \cite{EKQR} for the definition of actions and coactions of locally compact groups on $C^*$-cor\-re\-spond\-ences.) More precisely, if $(\gamma,\alpha)$ is an action of a locally compact group $G$ on a $C^*$-cor\-re\-spond\-ence $(X,A)$, one can form two constructions: an action $\beta$ of $G$ on $\mathcal{O}_X$  induced by $(\gamma,\alpha)$ \cite[Lemma~2.6.(b)]{HaoNg} on the one hand, and the crossed product correspondence $(X\rtimes_{\gamma,r}G,A\rtimes_{\alpha,r}G)$ of $(X,A)$ by $(\gamma,\alpha)$ (\cite[Proposition~3.2]{EKQR} or \cite{HaoNg}) on the other. It was shown in \cite{HaoNg} that if $G$ is amenable, then the crossed product by the action $\beta$ is isomorphic to the Cuntz-Pimsner algebra associated to $(X\rtimes_{\gamma}G,A\rtimes_{\alpha}G)$:
        \begin{equation}\label{Intro.1}
        \mathcal{O}_X\rtimes_\beta G\cong\mathcal{O}_{X\rtimes_\gamma G}.
        \end{equation}
    Similarly, it was shown in \cite{KQRo2} that a nondegenerate coaction $(\sigma,\delta)$ of a locally compact group $G$ on $(X,A)$ satisfying an invariance condition induces a coaction $\zeta$ of $G$ on $\mathcal{O}_X$ \cite[Proposition~3.1]{KQRo2}, and under the hypothesis of Cuntz-Pimsner covariance, the crossed product by $\zeta$ is again a Cuntz-Pimsner algebra \cite[Theorem~4.4]{KQRo2}:
        \begin{equation}\label{Intro.2}
        \mathcal{O}_X\rtimes_\zeta G\cong\mathcal{O}_{X\rtimes_\sigma G},
        \end{equation}
    where $X\rtimes_\sigma G$ is the $C^*$-cor\-re\-spond\-ence over $A\rtimes_\delta G$ arising from the coaction $(\sigma,\delta)$ \cite[Proposition 3.9]{EKQR}.

    The study in \cite{BS} proposed the framework of reduced Hopf $C^*$-algebras arising from multiplicative unitaries including both Kac algebras \cite{ES} and compact quantum groups \cite{Wo1,Wo2} (of course locally compact groups as well). The study also established the reduced crossed products of $C^*$-algebras by reduced Hopf $C^*$-algebra coactions, which are shown to be a natural generalization of crossed products by group actions and coactions. To each multiplicative unitary $V$, two reduced Hopf $C^*$-algebras $S_V$ and $\widehat{S}_V$ are associated in \cite{BS} under the regularity condition which was modified later in \cite{Wo4,SW} with manageability or modularity; in particular, the multiplicative unitaries of locally compact quantum groups \cite{KV} are known to be manageable. Thus the reduced Hopf $C^*$-algebras arising from multiplicative unitaries are a vast generalization of groups and their dual structures.

    The goal of this paper is to show that essentially the same results can be obtained if group actions or coactions studied in \cite{HaoNg,KQRo2} are replaced by Hopf $C^*$-algebra coactions. To this end, we first need a concept of a coaction of a Hopf $C^*$-algebra on a $C^*$-cor\-re\-spond\-ence. In \cite{BS0}, coaction of a Hopf $C^*$-algebra $S$ on a Hilbert $A$-module $X$ was introduced as a pair $(\sigma,\delta)$ of a linear map $\sigma:X\rightarrow M(X\otimes S)$ and a homomorphism $\delta:A\rightarrow M(A\otimes S)$ which are required to be, among other things, compatible with the Hilbert module structure of $X$. This notion was originally aimed to define equivariant KK-groups and generalize the Kasparov product in the setting of Hopf $C^*$-algebras. Since then, the notion of coactions on Hilbert modules has been extensively dealt with in various situations: for example \cite{Bui,Buss,BM,dCY,GZ,KQ0,TdC,Tom,Verg}. In this paper, we propose a definition of coaction of a Hopf $C^*$-algebra $S$ on a $C^*$-cor\-re\-spond\-ence $(X,A)$ as a coaction $(\sigma,\delta)$ of $S$ on the Hilbert $A$-module $X$ which is also compatible with the left action $\varphi_A$ (see Definition~\ref{DefofCoactions} for the precise definition), and show that this definition unifies the separate notions of group actions and nondegenerate group coactions on $(X,A)$  (Remark~\ref{Unify.act.coact.}). We then proceed to show that the passage from a group action or coaction on $(X,A)$ to an action or coaction on $\mathcal{O}_X$ can be generalized nicely in the Hopf $C^*$-algebra framework (Theorem~\ref{induced coactions on O_X}). When the Hopf $C^*$-algebra under consideration is a reduced one defined by a well-behaved multiplicative unitary in the sense of \cite{Timm}, we construct the reduced crossed product correspondences (Theorem~\ref{crossed product correspondences}), and prove an isomorphism result analogous to \eqref{Intro.1} and \eqref{Intro.2} under a suitable condition (Theorem~\ref{Main.Theorem.}). Applying our results we improve and extend the main results of \cite{HaoNg} and \cite{KQRo2} (Remark~\ref{Sec.5.Improve.KQRo2} and Corollary~\ref{Cor1.to.Main.Thm}).

    There have been plenty of works concerning ``natural'' coactions of compact quantum groups on the Cuntz algebra $\mathcal{O}_n$ with the focus on their fixed point algebras: for example, see \cite{Ga,KNW,M,Pao} among others. These coactions are the ones induced by coactions on the finite dimensional $C^*$-cor\-re\-spond\-ences $(\mathbb{C}^n,\mathbb{C})$, and actually, can be considered within a more general context of graph $C^*$-algebras \cite{KPRR,KPR,FLR}. In fact, we will extend in our upcoming paper the notion of labeling of a graph given in \cite{KQRa} to the setting of compact quantum groups. We will also consider coactions of compact quantum groups on any finite graphs. Natural coactions on $\mathcal{O}_n$ then can be viewed as the ones arising from labelings of the graph consisting of one vertex and $n$ edges, or alternatively, the ones arising from coactions on such a graph. Moreover, the crossed products by those natural coactions can be realized as Cuntz-Pimsner algebras. In light of these facts, it is natural and desirable to extend the works of \cite{HaoNg,KQRo2} from the point of view of Hopf $C^*$-algebra coactions.

    This paper is organized as follows.

    In Section~\ref{Sec.2}, we review basic facts from \cite[Chapter 1]{EKQR} and \cite[Appendix~A]{DKQ} on multiplier correspondences. We also collect from \cite{Kat,BS} definitions and facts on Cuntz-Pimsner algebras and reduced crossed products by Hopf $C^*$-algebra coactions on $C^*$-algebras. Note that in \cite{EKQR}, Hilbert $A$-$B$ bimodules were considered while we are concerned only with Hilbert $A$-$A$ bimodules, namely nondegenerate $C^*$-cor\-re\-spond\-ences $(X,A)$.

    In Section~\ref{Sec.3}, we define a coaction $(\sigma,\delta)$ of a Hopf $C^*$-algebra $S$ on a $C^*$-cor\-re\-spond\-ence $(X,A)$ (Definition~\ref{DefofCoactions}), and prove in Theorem~\ref{induced coactions on O_X} that if $(\sigma,\delta)$ is a coaction of $S$ on $(X,A)$ such that the ideal $J_X$ is weakly $\delta$-invariant (Definition~\ref{Sec.3.delta.invariant}), then $(\sigma,\delta)$ induces a coaction $\zeta$ of $S$ on the associated Cuntz-Pimsner algebra $\mathcal{O}_X$.

    Section~\ref{Sec.4} is devoted to constructing the reduced crossed product correspondence $(X\rtimes_\sigma\widehat{S},A\rtimes_\delta\widehat{S})$ from a coaction $(\sigma,\delta)$ of a reduced Hopf $C^*$-algebra $S$ defined by a well-behaved multiplicative unitary.

    In Section~\ref{Sec.5}, we prove an isomorphism analogous to \eqref{Intro.1} and \eqref{Intro.2} in the reduced Hopf $C^*$-algebra setting. Along the way we answer the question posed in \cite[Remark~4.5]{KQRo2}; specifically, we prove that Theorem 4.4 of \cite{KQRo2} still holds without the hypothesis of the Cuntz-Pimsner covariance for the canonical embedding of $(X,A)$ into the crossed product correspondence (see Remark~\ref{Sec.5.Improve.KQRo2}). The $C^*$-cor\-re\-spond\-ence $(X\rtimes_\sigma\widehat{S},A\rtimes_\delta\widehat{S})$ is shown to have a representation $(k_X\rtimes_\sigma{\rm id},k_A\rtimes_\delta{\rm id})$ on the reduced crossed product $\mathcal{O}_X\rtimes_\zeta\widehat{S}$ by the induced coaction $\zeta$ (Proposition~\ref{embedding XxShat to OxShat is a Toep.rep.}). We prove in Theorem \ref{Main.Theorem.} that
        \[\mathcal{O}_X\rtimes_\zeta\widehat{S}\cong\mathcal{O}_{X\rtimes_\sigma\widehat{S}}\]
    under the assumption that $(k_X\rtimes_\sigma{\rm id},k_A\rtimes_\delta{\rm id})$ is covariant. By applying this to group actions, we extend Theorem 2.10 of \cite{HaoNg} (see Corollary \ref{Cor1.to.Main.Thm}) to any locally compact groups. It is however, not so easy to determine whether the representation $(k_X\rtimes_\sigma{\rm id},k_A\rtimes_\delta{\rm id})$ is covariant or not without understanding the ideal $J_{X\rtimes_\sigma\widehat{S}}$ of $A\rtimes_\delta\widehat{S}$. Actually, $J_{X\rtimes_\sigma\widehat{S}}$ is not known even for the commutative case with some exceptions. For an action $(\gamma,\alpha)$ of a locally compact group $G$, it was shown that $J_{X\rtimes_\gamma G}=J_X\rtimes_\alpha G$ if $G$ is amenable (\cite[Proposition~2.7]{HaoNg}), which was the most difficult part in proving the main result of \cite{HaoNg} as was mentioned in the introductory section there. Recently, the same has been shown for a discrete group $G$ if $G$ is exact or if the action $\alpha$ has Exel's Approximation Property (\cite[Theorem~5.5]{BKQR}). However, we only know in general that $J_{X\rtimes_\sigma\widehat{S}}$ contains the ideal of $A\rtimes_\delta\widehat{S}$ generated by the image $\delta_\iota(J_X)$ (Proposition~\ref{Ext.Main.Thm.of.KQR2}). We bypass the difficulty regarding the ideal $J_{X\rtimes_\sigma\widehat{S}}$ by focusing our attention on the $(A\otimes\mathscr{K})$-multiplier correspondence $(M_{A\otimes\mathscr{K}}(X\otimes\mathscr{K}),M(A\otimes\mathscr{K}))$ in which the $C^*$-cor\-re\-spond\-ence $(X\rtimes_\sigma\widehat{S},A\rtimes_\delta\widehat{S})$ lies. This leads us to two equivalent conditions that the representation $(k_X\rtimes_\sigma{\rm id},k_A\rtimes_\delta{\rm id})$ is covariant (Theorem~\ref{Equiv.of.cov.}). From these equivalent conditions, we see that $(k_X\rtimes_\sigma{\rm id},k_A\rtimes_\delta{\rm id})$ is covariant if, in particular, $J_{X\rtimes_\sigma\widehat{S}}$ is generated by $\delta_\iota(J_X)$ or the left action $\varphi_A$ is injective (Corollary~\ref{Sec.5.Cor.varphi.inj.}), even though we do not know the ideal $J_{X\rtimes_\sigma\widehat{S}}$ explicitly.

    In Appendix~\ref{App.A}, we generalize \cite[Corollary~3.4]{APT} to our $C^*$-cor\-re\-spond\-ence setting, and then show that there exists a one-to-one correspondence between actions of a locally compact group $G$ and coactions of the commutative Hopf $C^*$-algebra $C_0(G)$ on a $C^*$-cor\-re\-spond\-ence. In Appendix~\ref{App.B}, we prove a $C^*$-cor\-re\-spond\-ence analogue to the well-known fact that $\mathcal{L}_A(A\otimes\mathcal{H})=M(A\otimes\mathcal{K}(\mathcal{H}))$ for a $C^*$-algebra $A$ and a Hilbert space $\mathcal{H}$. Using this, we recover from our construction of $(X\rtimes\widehat{S},A\rtimes\widehat{S})$ the crossed product correspondences $(X\rtimes_rG,A\rtimes_rG)$ for actions of locally compact groups $G$ given in \cite{EKQR}.

\section{Preliminaries}\label{Sec.2}

    In this section, we review some definitions and properties related to multiplier correspondences, Cuntz-Pimsner algebras, and reduced crossed products by reduced Hopf $C^*$-algebra coactions. Our references include \cite{EKQR,DKQ,Kat,BS}. 

    \subsection{$C^*$-correspondences}

    Let $A$ be a $C^*$-algebra. For two (right) Hilbert $A$-modules $X$ and $Y$, we denote by $\mathcal{L}(X,Y)=\mathcal{L}_A(X,Y)$ the Banach space of all adjointable operators from $X$ to $Y$, and by $\mathcal{K}(X,Y)=\mathcal{K}_A(X,Y)$ the closed subspace of $\mathcal{L}(X,Y)$ generated by the operators $\theta_{\xi,\eta}$:
        \[\theta_{\xi,\eta}(\eta')=\xi\cdot\langle\eta,\eta'\rangle_A\quad(\xi\in Y,\ \eta,\eta'\in X).\]
    We simply write $\mathcal{L}(X)$ and $\mathcal{K}(X)$ when $X=Y$; in this case $\mathcal{L}(X)$ becomes a maximal unital $C^*$-algebra containing $\mathcal{K}(X)$ as an essential ideal.

    A \emph{$C^*$-correspondence} over a $C^*$-algebra $A$ is a Hilbert $A$-module $X$ equipped with a homomorphism $\varphi_A:A\rightarrow\mathcal{L}(X)$, called the \emph{left action}. We use the notation $(X,A)$ of \cite{KQRo2} to refer to a $C^*$-correspondence $X$ over $A$. We say that $(X,A)$ is \emph{nondegenerate} if $\varphi_A$ is nondegenerate, that is, $\overline{\varphi_A(A)X}=X$.

    Every $C^*$-algebra $A$ has the natural structure of a nondegenerate $C^*$-cor\-re\-spond\-ence over itself, called the \emph{identity correspondence}, with the left action of identifying $A$ with $\mathcal{K}(A)$ (p.\ 368 of \cite{Kat}). When we regard a $C^*$-algebra as a $C^*$-cor\-re\-spond\-ence, we always mean this $C^*$-cor\-re\-spond\-ence. We also regard a Hilbert space $\mathcal{H}$ as a Hilbert $\mathbb{C}$-module such that the inner product is conjugate linear in the first variable; it is a $C^*$-cor\-re\-spond\-ence with $\varphi_{\mathbb{C}}(1)=1_{\mathcal{L}(\mathcal{H})}$.

    \subsection{Multiplier correspondences}

    Throughout the paper, we restrict ourselves to nondegenerate $C^*$-cor\-re\-spond\-ences, which in particular allows us to consider their multiplier correspondences that are a generalization of multiplier $C^*$-algebras.

    Let $(X,A)$ be a $C^*$-correspondence, and let $M(X):=\mathcal{L}(A,X)$. The \emph{multiplier correspondence} of $X$ is the $C^*$-cor\-re\-spond\-ence $M(X)$ over the multiplier algebra $M(A)$ with the Hilbert $M(A)$-module operations
        \begin{equation}\label{Prel.module.ops}
            m\cdot a:=m a,\quad\langle m,n\rangle_{M(A)}:=m^*n
        \end{equation}
    and the left action
        \begin{equation}\label{Prel.module.ops.2}
            \varphi_{M(A)}(a)m:=\overline{\varphi_A}(a)m
        \end{equation}
    for $m,n\in M(X)$ and $a\in M(A)$, where $\overline{\varphi_A}$ is the strict extension of the nondegenerate homomorphism $\varphi_A$ and $ma$, $m^*n$, and $\varphi_{M(A)}(a)m$ mean the compositions $m\circ a$, $m^*\circ n$, and $\varphi_{M(A)}(a)\circ m$, respectively. The identification of $X$ with $\mathcal{K}(A,X)$, in which each $\xi\in X$ is regarded as the operator $A\ni a\mapsto\xi\cdot a\in X$, gives an embedding of $X$ into $M(X)$, and we will always regard $X$ as a subspace of $M(X)$ through this embedding. Note that $\mathcal{K}(M(X))\subseteq M(\mathcal{K}(X))$ nondegenerately.

    The \emph{strict topology} on $M(X)$ is the locally convex topology such that a net $\{m_i\}$ in $M(X)$ converges strictly to 0 if and only if for $T\in\mathcal{K}(X)$ and $a\in A$, the nets $\{Tm_i\}$ and $\{m_i\cdot a\}$ both converge in norm to 0. It can be shown that $M(X)$ is the strict completion of $X$.

    Let $(X,A)$ and $(Y,B)$ be $C^*$-correspondences. A pair
        \[(\psi,\pi):(X,A)\rightarrow(M(Y),M(B))\]
    of a linear map $\psi:X\rightarrow M(Y)$ and a homomorphism $\pi:A\rightarrow M(B)$ is called a \emph{correspondence homomorphism} if
    \begin{itemize}
        \item[\rm(i)] $\psi(\varphi_A(a)\xi)=\varphi_{M(B)}(\pi(a))\,\psi(\xi)$ for $a\in A$ and $\xi\in X$;
        \item[\rm(ii)] $\pi(\langle\xi,\eta\rangle_A)=\langle\psi(\xi),\psi(\eta)\rangle_{M(B)}$ for $\xi,\eta\in X$.
    \end{itemize}
    It is automatic that $\psi(\xi\cdot a)=\psi(\xi)\cdot\pi(a)$ (see the comment below \cite[Definition~2.3]{Kat2}). We say that $(\psi,\pi)$ is \emph{injective} if $\pi$ is injective; if so $\psi$ is isometric. We also say that $(\psi,\pi)$ is \emph{nondegenerate} if
        $\overline{\psi(X)\cdot B}=Y$ and $\overline{\pi(A)B}=B$.
    In this case, $(\psi,\pi)$ extends uniquely to a strictly continuous correspondence homomorphism
        \[(\overline{\psi},\overline{\pi}):(M(X),M(A))\rightarrow(M(Y),M(B))\]
    (\cite[Theorem 1.30]{EKQR}). Note that if $(\psi,\pi)$ is injective, then so is $(\overline{\psi},\overline{\pi})$.

    A correspondence homomorphism $(\psi,\pi):(X,A)\rightarrow(M(Y),M(B))$ determines a (unique) homomorphism
        $\psi^{(1)}:\mathcal{K}(X)\rightarrow\mathcal{K}(M(Y))\subseteq M(\mathcal{K}(Y))$
    such that
        \[\psi^{(1)}(\theta_{\xi,\eta})=\psi(\xi)\psi(\eta)^*\quad(\xi,\eta\in X)\]
    (see for example \cite[Definition 2.4]{Kat2} and the comment below it). If $(\psi,\pi)$ is nondegenerate, then so is $\psi^{(1)}$; it is straightforward to verify that
        \begin{equation}\label{Pre.psi.1}
            \psi(T\xi)=\overline{\psi^{(1)}}(T)\psi(\xi),\quad\overline{\psi^{(1)}}(mn^*)=\overline{\psi}(m)\overline{\psi}(n)^*
        \end{equation}
    for $T\in\mathcal{L}(X)$, $\xi\in X$, and $m,n\in M(X)$. The first relation of \eqref{Pre.psi.1} shows that $\overline{\psi^{(1)}}$ is injective whenever $\psi$ is injective.

    \subsection{Tensor product correspondences}

    In this paper, the tensor products of Hilbert modules always mean the exterior ones (\cite[pp.\ 34--35]{Lance}). The tensor products of $C^*$-algebras are the minimal ones.

    Let $(X_1,A_1)$ and $(X_2,A_2)$ be $C^*$-cor\-re\-spon\-dences. We will freely use the following identification
        \[\mathcal{K}(X_1\otimes X_2)=\mathcal{K}(X_1)\otimes\mathcal{K}(X_2)\]
    via $\theta_{\xi_1\otimes\xi_2,\,\eta_1\otimes\eta_2}=\theta_{\xi_1,\,\eta_1}\otimes\theta_{\xi_2,\,\eta_2}$. Equipped with the left action
        $\varphi_{A_1\otimes A_2}=\varphi_{A_1}\otimes\varphi_{A_2},$
    the tensor product $X_1\otimes X_2$ then becomes a $C^*$-cor\-re\-spon\-dence over $A_1\otimes A_2$, called the \emph{tensor product correspondence}.

    Let $(\psi_i,\pi_i):(X_i,A_i)\rightarrow(M(Y_i),M(B_i))$ $(i=1,2)$ be correspondence homomorphisms. Then there exists a unique correspondence homomorphism
         \[(\psi_1\otimes\psi_2,\pi_1\otimes\pi_2):(X_1\otimes X_2,A_1\otimes A_2)\rightarrow(M(Y_1\otimes Y_2),M(B_1\otimes B_2))\]
    such that $(\psi_1\otimes\psi_2)(\xi_1\otimes\xi_2)=\psi_1(\xi_1)\otimes\psi_2(\xi_2)$. If both $(\psi_i,\pi_i)$ are nondegenerate then $(\psi_1\otimes\psi_2,\pi_1\otimes\pi_2)$ is also nondegenerate (\cite[Proposition 1.38]{EKQR}).


    \subsection{Cuntz-Pimsner algebras}

    Let $(X,A)$ be a $C^*$-cor\-re\-spond\-ence, and let
        \[J_X:=\varphi_A^{-1}(\mathcal{K}(X))\cap\{a\in A:ab=0\ {\rm for}\ b\in\ker\varphi_A\}.\]
    Then $J_X$ is characterized as the largest ideal of $A$ which is mapped injectively into $\mathcal{K}(X)$ by $\varphi_A$. A correspondence homomorphism $(\psi,\pi):(X,A)\rightarrow(B,B)$ into an identity correspondence $(B,B)$ is called a \emph{representation} of $(X,A)$ on $B$ and denoted simply by $(\psi,\pi):(X,A)\rightarrow B$. We say that $(\psi,\pi)$ is \emph{covariant} if
        \[\psi^{(1)}(\varphi_A(a))=\pi(a)\quad(a\in J_X)\]
    (\cite[Definition 3.4]{Kat}). We denote by $(k_X,k_A)$ the universal covariant representation of $(X,A)$ which is known to be injective (\cite[Proposition~4.9]{Kat}). The \emph{Cuntz-Pimsner algebra} $\mathcal{O}_X$ is the $C^*$-algebra generated by $k_X(X)$ and $k_A(A)$. Note that the embedding $k_A:A\hookrightarrow\mathcal{O}_X$ is nondegenerate by our standing assumption that $(X,A)$ is nondegenerate. From the universality of $(k_X,k_A)$, if $(\psi,\pi)$ is a covariant representation of $(X,A)$ on $B$, there exists a unique homomorphism $\psi\times\pi:\mathcal{O}_X\rightarrow B$ called the \emph{integrated form} of $(\psi,\pi)$ such that $\psi=(\psi\times\pi)\circ k_X$ and $\pi=(\psi\times\pi)\circ k_A$.

    A representation $(\psi,\pi)$ of $(X,A)$ is said to \emph{admit a gauge action} if there exists an action $\beta$ of the unit circle $\mathbb{T}$ on the $C^*$-subalgebra generated by $\psi(X)$ and $\pi(A)$ such that $\beta_z(\psi(\xi))=z\psi(\xi)$ and $\beta_z(\pi(a))=\pi(a)$ for $z\in\mathbb{T}$, $\xi\in X$, and $a\in A$. The universal covariant representation $(k_X,k_A)$ clearly admits a gauge action. The \emph{gauge invariant uniqueness theorem} \cite[Theorem 6.4]{Kat} asserts that an injective covariant representation $(\psi,\pi)$ admits a gauge action if and only if $\psi\times\pi$ is injective.

    \subsection{$C$-multiplier correspondences}

    We recall from \cite[Appendix~A]{DKQ} basic definitions and facts on $C$-multiplier correspondences. We also fix some notations and provide results that will be used in the subsequent sections.

    Let $(X,A)$ be a $C^*$-correspondence, $C$ be a $C^*$-algebra, and $\kappa:C\rightarrow M(A)$ be a nondegenerate homomorphism. The \emph{$C$-multiplier correspondence} $M_C(X)$ of $X$ and the \emph{$C$-multiplier algebra} $M_C(A)$ of $A$ are defined by
        \[M_C(X):=\{m\in M(X):\varphi_{M(A)}(\kappa(C))m\cup m\cdot\kappa(C)\subseteq X\},\]
        \[M_C(A):=\{a\in M(A):\kappa(C)a\cup a\kappa(C)\subseteq A\}.\]
    Under the restriction of the operations \eqref{Prel.module.ops} and \eqref{Prel.module.ops.2}, $(M_C(X),M_C(A))$ becomes a $C^*$-cor\-re\-spond\-ence (\cite[Lemma A.9.(2)]{DKQ}).

    \begin{notation}\rm
        We mean by $M_A(X)$ the $A$-multiplier correspondence
            \[M_A(X)=\big\{m\in M(X):\varphi_A(A)m\subseteq X\big\}\]
        determined by $\kappa={\rm id}_A$, and by $M_A(\mathcal{K}(X))$ the $A$-multiplier algebra
            \[M_A(\mathcal{K}(X))=\{m\in M(\mathcal{K}(X)):\varphi_A(A)m\cup m\varphi_A(A)\subseteq\mathcal{K}(X)\}\]
        determined by the left action $\varphi_A$.
    \end{notation}

    Note that $\mathcal{K}(M_A(X))\subseteq M_A(\mathcal{K}(X))$ (\cite[Lemma A.9.(3)]{DKQ}).

    The \emph{$C$-strict topology} on $M_C(X)$ is the locally convex topology whose neighborhood system at 0 is generated by the family $\{m:\|\varphi_{M(A)}(\kappa(c))m\|\leq\epsilon\}$ and $\{m:\|m\cdot\kappa(c)\|\leq\epsilon\}$ ($c\in C$, $\epsilon>0$). The $C$-strict topology is stronger than the relative strict topology on $M_C(X)$, and $M_C(X)$ is the $C$-strict completion of $X$. Likewise, the \emph{$C$-strict topology} on $M_C(A)$ is the locally convex topology defined by the family of seminorms $\|\kappa(c)\cdot\|+\|\cdot\kappa(c)\|$ ($c\in C$).

    \begin{rmk}\label{Prel.C-strict.top.}\rm
        Let $(X,A)$ be a $C^*$-correspondence and $M_{C_i}(X)$ be the $C_i$-multiplier correspondence determined by a nondegenerate homomorphism $\kappa_i:C_i\rightarrow M(A)$ ($i=1,2$). It is clear that if $\kappa_1(C_1)$ is nondegenerately contained in $M(\kappa_2(C_2))(\subseteq M(A))$, then $M_{C_1}(X)\subseteq M_{C_2}(X)$ and the $C_1$-strict topology on $M_{C_1}(X)$ is stronger than the relative $C_2$-strict topology. In particular, $M_C(X)\subseteq M_A(X)$ and the $C$-strict topology is stronger than the relative $A$-strict topology.
    \end{rmk}

    For a not necessarily nondegenerate correspondence homomorphism, we still have an extension by \cite[Proposition A.11]{DKQ}. Let $(\psi,\pi):(X,A)\rightarrow(M_D(Y),M_D(B))$ be a correspondence homomorphism, where $(M_D(Y),M_D(B))$ is a $D$-mul\-ti\-pli\-er correspondence determined by a nondegenerate homomorphism $\kappa_D:D\rightarrow M(B)$. Assume that $\kappa_C:C\rightarrow M(A)$ and $\lambda:C\rightarrow M(\kappa_D(D))(\subseteq M(B))$ are nondegenerate homomorphisms such that $\pi(\kappa_C(c)a)=\lambda(c)\pi(a)$ for $c\in C$ and $a\in A$. Then $(\psi,\pi)$ extends uniquely to a  $C$-strict to $D$-strictly continuous correspondence homomorphism
        $(\overline{\psi},\overline{\pi}):(M_C(X),M_C(A))\rightarrow (M_D(Y),M_D(B)),$
    where $(M_C(X),M_C(A))$ is the $C$-multiplier correspondence determined by $\kappa_C$.

    \begin{rmks}\label{Prel.Str.Ext.}\rm
        (1) If $(\psi,\pi)$ is nondegenerate, then every $C$-strict to $D$-strictly continuous extension of $(\psi,\pi)$ coincides with the restriction of its usual strict extension.

        (2) Suppose that $\overline{\psi}_i:M_{C_i}(X)\rightarrow M_{D_i}(Y)$ are $C_i$-strict to $D_i$-strictly continuous extensions $(i=1,2)$. If $M_{C_1}(X)\subseteq M_{C_2}(X)$ and $M_{D_1}(Y)\subseteq M_{D_2}(Y)$ and if the $C_1$-strict and $D_1$-strict topologies are stronger than the relative $C_2$-strict and $D_2$-strict topologies, respectively, then $\overline{\psi}_1=\overline{\psi}_2|_{M_{C_1}(X)}$.
    \end{rmks}

    We will frequently need the following special form of \cite[Proposition A.11]{DKQ}.

    \begin{thm}[{\cite[Corollary A.14]{DKQ}}]\label{DKQ,Cor.A.14}
        Let $(\psi,\pi):(X,A)\rightarrow B$ be a representation with $\pi$ nondegenerate. Then

        {\rm(i)} $(\psi,\pi)$ extends uniquely to an $A$-strictly continuous correspondence homomorphism
                \[(\overline{\psi},\overline{\pi}):(M_A(X),M(A))\rightarrow M_A(B),\]
                where $M_A(B)$ is the $A$-multiplier algebra determined by $\pi$.

        {\rm(ii)} $\psi^{(1)}:\mathcal{K}(X)\rightarrow B$ extends uniquely to an $A$-strictly continuous homomorphism $\overline{\psi^{(1)}}:M_A(\mathcal{K}(X))\rightarrow M_A(B)$; moreover, \[\overline{\psi^{(1)}}=\overline{\psi}^{\,(1)}\] on $\mathcal{K}(M_A(X))$, that is, $\overline{\psi^{(1)}}(mn^*)=\overline{\psi}(m)\overline{\psi}(n)^*$ for $m,n\in M_A(X)$.
    \end{thm}

    \begin{notation}\rm
        Let $(X,A)$ be a $C^*$-correspondence and $C$ be a $C^*$-algebra. Consider the representation $(k_X\otimes{\rm id}_C,k_A\otimes{\rm id}_C):(X\otimes C,A\otimes C)\rightarrow\mathcal{O}_X\otimes C$. Since $k_A\otimes{\rm id}_C$ is nondegenerate, $k_X\otimes{\rm id}_C$ extends to the $(A\otimes C)$-strictly continuous map
            \begin{equation*}\label{Prel.Str.Ext.k.times.id}
                \overline{k_X\otimes{\rm id}_C}:M_{A\otimes C}(X\otimes C)\rightarrow M_{A\otimes C}(\mathcal{O}_X\otimes C)
            \end{equation*}
        by Theorem \ref{DKQ,Cor.A.14}.(i). Throughout the paper, we mean by $\overline{k_X\otimes{\rm id}_C}$ this extension, and by $M_{A\otimes C}(\mathcal{O}_X\otimes C)$ the $(A\otimes C)$-multiplier algebra determined by $k_A\otimes{\rm id}_C$. On the other hand, $(M_C(X\otimes C),M_C(A\otimes C))$ is the $C$-multiplier correspondence determined by the embedding $C\hookrightarrow M(A\otimes C)$ onto the last factor.
    \end{notation}

    For an ideal $I$ of a $C^*$-algebra $B$, let
        \[M(B;I):=\{m\in M(B): mB\cup Bm\subseteq I\}.\]
    By \cite[Lemma 2.4.(i)]{KQRo1}, $M(B;I)$ is the strict closure of $I$ in $M(B)$.

    \begin{lem}\label{Lemma.for.Ext.Main.Thm.of.KQR2 0}
        Let $(X,A)$ be a $C^*$-correspondence. Then the ideal $J_{M_A(X)}$ is contained in the strict closure of $J_X$, that is,
            \[J_{M_A(X)}\subseteq M(A;J_X).\]
    \end{lem}

    \begin{proof}
        We need to show that the ideal $AJ_{M_A(X)}$ is contained in $J_X$. By definition, we have
            \[\varphi_A(AJ_{M_A(X)})\subseteq\varphi_A(A)\mathcal{K}(M_A(X))\subseteq\varphi_A(A)M_A(\mathcal{K}(X))\subseteq\mathcal{K}(X).\]
        We also have
            \[AJ_{M_A(X)}\ker\varphi_A\subseteq J_{M_A(X)}\ker\varphi_{M(A)}=0.\]
        Consequently, $AJ_{M_A(X)}\subseteq J_X$.
    \end{proof}

    The next lemma, contained in the proof of \cite[Lemma 2.5]{KQRo2}, will be useful in proving Theorem \ref{induced coactions on O_X}, Proposition \ref{Ext.Main.Thm.of.KQR2}, and Theorem \ref{Equiv.of.cov.}.

    \begin{lem}\label{Lemma.for.Ext.Main.Thm.of.KQR2}
        Let $(X,A)$ be a $C^*$-correspondence and $C$ be a $C^*$-algebra. Then
        \begin{equation}\label{Covariancy.of.the.Strict.Ext.1}
            \overline{(k_X\otimes{\rm id}_C)^{(1)}}\circ\varphi_{M(A\otimes C)}=\overline{k_A\otimes{\rm id}_C}
        \end{equation}
        holds on $M(A\otimes C;J_X\otimes C)$, that is, the diagram
        \[ \xymatrix{ M_{A\otimes C}(\mathcal{K}(X\otimes C)) \ar[drr]^-{\qquad\overline{(k_X\otimes{\rm id}_C)^{(1)}}} && \\
            M(A\otimes C;J_X\otimes C) \ar[u]^-{\varphi_{M(A\otimes C)}} \ar[rr]_-{\overline{k_A\otimes{\rm id}_C}} &&
            M_{A\otimes C}(\mathcal{O}_X\otimes C)
        }\]
        commutes.
    \end{lem}

    \begin{proof}
        By definition, the vertical map makes sense and is $(A\otimes C)$-strictly continuous. Also, Theorem \ref{DKQ,Cor.A.14}.(ii) says that $(k_X\otimes{\rm id}_C)^{(1)}$ extends $(A\otimes C)$-strictly to the homomorphism $\overline{(k_X\otimes{\rm id}_C)^{(1)}}$ indicated by the lower right arrow. Hence the composition on the left side of \eqref{Covariancy.of.the.Strict.Ext.1} is well-defined on $M(A\otimes C;J_X\otimes C)$ and $(A\otimes C)$-strictly continuous. On the other hand, the horizontal map is the restriction of the usual strict extension $\overline{k_A\otimes{\rm id}_C}$ and $(A\otimes C)$-strictly continuous. Since \eqref{Covariancy.of.the.Strict.Ext.1} is valid on $J_X\otimes C$, the conclusion now follows by $(A\otimes C)$-strict continuity and the fact that $J_X\otimes C$ is $(A\otimes C)$-strictly dense in $M(A\otimes C;J_X\otimes C)$.
    \end{proof}

    Recall from \cite[Definition 12.4.3]{BO} the following terminology. Let $A$ and $C$ be $C^*$-algebras and $J$ be a closed subspace of $A$. The triple $(J,A,C)$ is said to \emph{satisfy the slice map property} if the space
        \[ F(J,A,C)=\{x\in A\otimes C: ({\rm id}\otimes\omega)(x)\in J\ {\rm for}\ \omega\in C^*\} \]
    equals the norm closure $J\otimes C$ of the algebraic tensor product $J\odot C$ in $A\otimes C$.

    \begin{rmks}\label{S.M.P.}\rm
        (1) If $J$ is an ideal of $A$, then $(J,A,C)$ satisfies the slice map property if and only if the sequence
            \[ 0\longrightarrow J\otimes C\longrightarrow A\otimes C\longrightarrow (A/J)\otimes C \longrightarrow0 \]
        is exact; this is the case if $A$ is locally reflexive or $C$ is exact (see below \cite[Definition~12.4.3]{BO}).

        (2) Let $\mathcal{H}$ be a Hilbert space. If $C$ is a $C^*$-subalgebra of $\mathcal{L}(\mathcal{H})$, then $F(J,A,C)$ equals the norm closure of the following space
            \[ \{x\in A\otimes C: ({\rm id}\otimes\omega)(x)\in J\ {\rm for}\ \omega\in\mathcal{L}(\mathcal{H})_*\}. \]
    \end{rmks}

    \begin{cor}\label{Sec.3.prop.for.thm.in.Sec.5}
        Let $(X,A)$ be a $C^*$-correspondence and $C$ be a $C^*$-algebra. Suppose that $(J_X,A,C)$ satisfies the slice map property. Then
            \[J_{X\otimes C}=J_X\otimes C.\]
        Furthermore,
            \[J_{M_{A\otimes C}(X\otimes C)}\subseteq M(A\otimes C;J_X\otimes C)\]
        and the injective representation
        \[(\overline{k_X\otimes{\rm id}_C},\overline{k_A\otimes{\rm id}_C}):(M_{A\otimes C}(X\otimes C),M(A\otimes C))\rightarrow M_{A\otimes C}(\mathcal{O}_X\otimes C)\]
        is covariant.
    \end{cor}

    \begin{proof}
        We always have $J_{X\otimes C}\supseteq J_X\otimes C$ as shown in the first part of the proof of \cite[Lemma 2.6]{KQRo2}. We thus only need to show the converse $J_{X\otimes C}\subseteq F(J_X,A,C)=J_X\otimes C$. But, this can be done in the same way as the second part of the proof of \cite[Lemma 2.6]{KQRo2}, and then the first assertion of the corollary follows. Lemma~\ref{Lemma.for.Ext.Main.Thm.of.KQR2 0} then verifies the second assertion on the inclusion. Finally, since $\varphi_{M(A\otimes C)}$ maps $J_{M_{A\otimes C}(X\otimes C)}$ into $\mathcal{K}(M_{A\otimes C}(X\otimes C))$ on which
            \[\overline{(k_X\otimes{\rm id})^{(1)}}=\overline{k_X\otimes{\rm id}_C}^{\,(1)}\]
        by Theorem \ref{DKQ,Cor.A.14}.(ii), the representation is covariant by Lemma \ref{Lemma.for.Ext.Main.Thm.of.KQR2}.
    \end{proof}

    \begin{cor}\label{O_(XotimesB)=O_XotimesB}
        Under the same hypothesis of Corollary \ref{Sec.3.prop.for.thm.in.Sec.5}, the injective representation
            \[(k_X\otimes{\rm id}_C,k_A\otimes{\rm id}_C):(X\otimes C,A\otimes C)\rightarrow\mathcal{O}_X\otimes C\]
        is covariant and the integrated form $(k_X\otimes{\rm id}_C)\times(k_A\otimes{\rm id}_C):\mathcal{O}_{X\otimes C}\rightarrow\mathcal{O}_X\otimes C$ is a surjective isomorphism.
    \end{cor}

    \begin{proof}
        Generally we have $J_Y\subseteq J_{M_B(Y)}$ for a $C^*$-cor\-re\-spond\-ence $(Y,B)$ since $J_Y$ is an ideal of $M(B)$ and is mapped injectively into $\mathcal{K}(Y)\subseteq\mathcal{K}(M_B(Y))$ by $\varphi_{M(B)}$. Hence $J_{X\otimes C}\subseteq J_{M_{A\otimes C}(X\otimes C)}$, and therefore $(k_X\otimes{\rm id}_C,k_A\otimes{\rm id}_C)$ is covariant by Corollary \ref{Sec.3.prop.for.thm.in.Sec.5}. The integrated form is clearly surjective. Since $(k_X,k_A)$ admits a gauge action, and hence so does $(k_X\otimes{\rm id}_C,k_A\otimes{\rm id}_C)$, the integrated form must be injective by \cite[Theorem 6.4]{Kat}.
    \end{proof}

    \subsection{Reduced and dual reduced Hopf $C^*$-algebras}

    By a \emph{Hopf $C^*$-algebra} we always mean a bisimplifiable Hopf $C^*$-algebra in the sense of \cite{BS}, that is, a pair $(S,\Delta)$ of a $C^*$-algebra $S$ and a nondegenerate homomorphism $\Delta:S\rightarrow M(S\otimes S)$ called the \emph{comultiplication} of $S$ satisfying
    \begin{itemize}
        \item[(i)] $\overline{\Delta\otimes{\rm id}}\circ\Delta=\overline{{\rm id}\otimes\Delta}\circ\Delta$;
        \item[(ii)] $\overline{\Delta(S)(1_{M(S)}\otimes S)}=S\otimes S=\overline{\Delta(S)(S\otimes1_{M(S)})}$.
    \end{itemize}
    Let $G$ be a locally compact group. Then $(C_0(G),\Delta_G)$ is a Hopf $C^*$-algebra with the comultiplication $\Delta_G(f)(r,s)=f(rs)$ for $f\in C_0(G)$ and $r,s\in G$. The full group $C^*$-algebra $C^*(G)$ equipped with the comultiplication given by $r\mapsto r\otimes r$ for $r\in G$ is also a Hopf $C^*$-algebra. The same is true for the reduced group $C^*$-algebra $C^*_r(G)$ such that the canonical surjection $\lambda:C^*(G)\rightarrow C^*_r(G)$ is a morphism in the sense of \cite{BS} (also see \cite[Example 4.2.2]{Timm}).

    Let $\mathcal{H}$ be a Hilbert space. A unitary operator $V$ acting on $\mathcal{H}\otimes\mathcal{H}$ is said to be \emph{multiplicative} if it satisfies the pentagonal relation $V_{12}V_{13}V_{23}=V_{23}V_{12}$, where we use the leg-numbering notations $V_{ij}$ such that $V_{12}\in\mathcal{L}(\mathcal{H}\otimes\mathcal{H}\otimes\mathcal{H})$ denotes the unitary $V\otimes1$ for example (see \cite[p.\ 428]{BS}). For each functional $\omega\in\mathcal{L}(\mathcal{H})_*$, define the operators $L(\omega)$ and $\rho(\omega)$ in $\mathcal{L}(\mathcal{H})$ by
        \[L(\omega)=\overline{\omega\otimes{\rm id}}(V),\quad\rho(\omega)=\overline{{\rm id}\otimes\omega}(V),\]
    where the maps $\overline{\omega\otimes{\rm id}}$ and $\overline{{\rm id}\otimes\omega}$ denote the usual strict extension to the multiplier algebra $M(\mathcal{K}(\mathcal{H})\otimes\mathcal{K}(\mathcal{H}))(=\mathcal{L}(\mathcal{H}\otimes\mathcal{H}))$. The \emph{reduced algebra} $S_V$ and the \emph{dual reduced algebra} $\widehat{S}_V$ are defined as the following norm closed subspaces of $\mathcal{L}(\mathcal{H})$:
        \[S_V=\overline{\{L(\omega):\omega\in\mathcal{L}(\mathcal{H})_*\}},\quad
        \widehat{S}_V=\overline{\{\rho(\omega):\omega\in\mathcal{L}(\mathcal{H})_*\}}.\]
    They are known to be nondegenerate subalgebras of $\mathcal{L}(\mathcal{H})$ (\cite[Prop\-o\-si\-tion~1.4]{BS}).

    A multiplicative unitary $V$ acting on $\mathcal{H}\otimes\mathcal{H}$ is said to be \emph{well-behaved} if both $S_V$ and $\widehat{S}_V$ are Hopf $C^*$-algebras with the comultiplications $\Delta_V(s)=V(s\otimes1)V^*$ and $\widehat{\Delta}_V(x)=V^*(1\otimes x)V$
    for $s\in S$ and $x\in\widehat{S}$, and $V\in M(\widehat{S}\otimes S)$ (\cite[Definition~7.2.6.i)]{Timm}).

    \begin{rmk}\rm
        When we consider a well-behaved multiplicative unitary $V$, we will not need the last property $V\in M(\widehat{S}\otimes S)$. In fact, we only need the property that $V$ gives rise to two Hopf $C^*$-algebras $S_V$ and $\widehat{S}_V$. 
        It should be stressed though that many important and significant Hopf $C^*$-algebras come from well-behaved multiplicative unitaries the class of which includes those with regularity \cite{BS}, manageability \cite{Wo4} and modularity \cite{SW}. In particular, locally compact quantum groups \cite{KV} are the Hopf $C^*$-algebras arising from well-behaved multiplicative unitaries.
    \end{rmk}

    For a locally compact group $G$, let $W_G$ and $\widehat{W}_G$ be the regular multiplicative unitaries acting on $L^2(G)\otimes L^2(G)$ by
        \[(W_G\xi)(r,s)=\xi(r,r^{-1}s),\quad(\widehat{W}_G\xi)(r,s)=\xi(sr,s)\]
    for $\xi\in C_c(G\times G)$ and $r,s\in G$. It can be shown that $S_{W_G}=C^*_r(G)=\widehat{S}_{\widehat{W}_G}$ as Hopf $C^*$-algebras. Let $\mu_G$ and $\check{\mu}_G$ be the nondegenerate embeddings $C_0(G)\hookrightarrow\mathcal{L}(L^2(G))$ given by
    \begin{equation}\label{Prel.pi.and.U}
        \big(\mu_G(f)h\big)(r)=f(r)h(r),\quad\big(\check{\mu}_G(f)h\big)(r)=f(r^{-1})h(r)
    \end{equation}
    for $h\in C_c(G)$. Then $\mu_G$ and $\check{\mu}_G$ are isomorphisms from the Hopf $C^*$-algebra $C_0(G)$ onto $\widehat{S}_{W_G}$ and $S_{\widehat{W}_G}$, respectively (see for example \cite[Example 9.3.11]{Timm}).

\subsection{Reduced crossed products}

    By a \emph{coaction} of a Hopf $C^*$-algebra $(S,\Delta)$ on a $C^*$-algebra $A$ we always mean a nondegenerate homomorphism $\delta:A\rightarrow M(A\otimes S)$ such that
    \begin{itemize}
        \item[(i)] $\delta$ satisfies the \emph{coaction identity} $\overline{\delta\otimes{\rm id}}\circ\delta=\overline{{\rm id}\otimes\Delta}\circ\delta$;
        \item[(ii)] $\delta$ satisfies the \emph{coaction nondegeneracy} $\overline{\delta(A)(1_{M(A)}\otimes S)}=A\otimes S$.
    \end{itemize}

    Let $V$ be a well-behaved multiplicative unitary acting on $\mathcal{H}\otimes\mathcal{H}$. Let $\delta$ be a coaction of the reduced Hopf $C^*$-algebra $S_V$ on $A$, and $\iota_{S_V}:S_V\hookrightarrow M(\mathcal{K}(\mathcal{H}))$ be the inclusion map. We denote by $\delta_\iota$ the following composition
        \begin{equation}\label{Prel.delta.iota}
            \delta_\iota:=\overline{{\rm id}_A\otimes\iota_{S_V}}\circ\delta:A\rightarrow M(A\otimes\mathcal{K}(\mathcal{H})).
        \end{equation}
    The \emph{reduced crossed product} $A\rtimes_\delta\widehat{S}_V$ of $A$ by the coaction $\delta$ of $S_V$ is defined to be the following norm closed subspace of $M(A\otimes\mathcal{K}(\mathcal{H}))$
        \[A\rtimes_\delta\widehat{S}_V=\overline{\delta_\iota(A)(1_{M(A)}\otimes\widehat{S}_V)},\]
    where $1_{M(A)}\otimes\widehat{S}_V$ denotes the image of the canonical embedding $\widehat{S}_V\hookrightarrow M(A\otimes\mathcal{K}(\mathcal{H}))$. By \cite[Lemma 7.2]{BS}, $A\rtimes_\delta\widehat{S}_V$ is a $C^*$-algebra.

    \begin{rmk}\rm
        In the literature, the reduced crossed product $A\rtimes_\delta\widehat{S}_V$ is usually defined as a subalgebra of $\mathcal{L}_A(A\otimes\mathcal{H})$ which can be identified with $M(A\otimes\mathcal{K}(\mathcal{H}))$. For the arguments concerning multiplier correspondences and the relevant strict topologies, it seems to be more convenient to work with $M(A\otimes\mathcal{K}(\mathcal{H}))$ rather than $\mathcal{L}_A(A\otimes\mathcal{H})$. This leads us to regard $A\rtimes_\delta\widehat{S}_V$ as a subalgebra of $M(A\otimes\mathcal{K}(\mathcal{H}))$.
    \end{rmk}

    Let $G$ be a locally compact group and $A$ be a $C^*$-algebra. It is well-known that there exists a one-to-one correspondence between actions of $G$ on $A$ and coactions of $C_0(G)$ on $A$ such that each action $\alpha$ determines a coaction $\delta^\alpha$ given by $\delta^\alpha(a)(r)=\alpha_r(a)$ for $a\in A$ and $r\in G$. Moreover, if $\alpha:G\rightarrow Aut(A)$ is an action then the reduced crossed product $A\rtimes_{\alpha,r}G$ coincides with the crossed product $A\rtimes_{\delta^\alpha_G}\widehat{S}_{\widehat{W}_G}$ by the coaction
        \[\delta^\alpha_G=\overline{{\rm id}_A\otimes\check{\mu}_G}\circ\delta^\alpha:A\rightarrow M(A\otimes S_{\widehat{W}_G})\]
    when viewed as subalgebras of $M(A\otimes\mathcal{K}(L^2(G)))$ (see for example \cite[Chapter~9]{Timm}). 

    A \emph{nondegenerate coaction} of $G$ on a $C^*$-algebra $A$ is an injective coaction $\delta$ of the Hopf $C^*$-algebra $C^*(G)$ on $A$ (\cite[Definition A.21]{EKQR}). Let
        \begin{equation}\label{Prel.coaction.of.G}
            \delta_\lambda:=\overline{{\rm id}_A\otimes\lambda}\circ\delta:A\rightarrow M(A\otimes C^*_r(G))=M(A\otimes S_{W_G}).
        \end{equation}
    The crossed product $A\rtimes_\delta G$ by $\delta$ is defined to be the reduced crossed crossed product $A\rtimes_{\delta_\lambda}\widehat{S}_{W_G}$ by $\delta_\lambda$ (\cite[Definition~A.39]{EKQR}).

\section{Coactions of Hopf $C^*$-algebras on $C^*$-correspondences}\label{Sec.3}

    In this section, we define a coaction $(\sigma,\delta)$ of a Hopf $C^*$-algebra on a $C^*$-cor\-re\-spond\-ence $(X,A)$, and prove that $(\sigma,\delta)$ induces a coaction on the associated Cuntz-Pimsner algebra $\mathcal{O}_X$ under a certain invariance condition (Theorem~\ref{induced coactions on O_X}). Recall that the $C^*$-cor\-re\-spon\-dences considered in this paper are always nondegenerate.

    \begin{defn}\label{DefofCoactions}\rm
        A \emph{coaction} of a Hopf $C^*$-algebra $(S,\Delta)$ on a $C^*$-cor\-re\-spond\-ence $(X,A)$ is a nondegenerate correspondence homomorphism
            \[(\sigma,\delta):(X,A)\rightarrow(M(X\otimes S),M(A\otimes S))\]
        such that
        \begin{itemize}
            \item[\rm(i)] $\delta$ is a coaction of $S$ on the $C^*$-algebra $A$;
            \item[\rm(ii)] $\sigma$ satisfies the \emph{coaction identity} $\overline{\sigma\otimes{\rm id}_S}\circ\sigma=\overline{{\rm id}_X\otimes\Delta}\circ\sigma$;
            \item[\rm(iii)] $\sigma$ satisfies the \emph{coaction nondegeneracy}
                \[ \overline{\varphi_{M(A\otimes S)}(1_{M(A)}\otimes S)\,\sigma(X)}=X\otimes S.\]
        \end{itemize}
    \end{defn}

    Note that the strict extensions $\overline{\sigma\otimes{\rm id}_S}$ and $\overline{{\rm id}_X\otimes\Delta}$ in (ii) are well-defined because the tensor product of two nondegenerate correspondence homomorphisms is also nondegenerate (\cite[Proposition 1.38]{EKQR}).

    \begin{rmks}\rm
        (1) It should be noted that
            \[\overline{\sigma(X)\cdot(1_{M(A)}\otimes S)}=X\otimes S.\]
        This follows by the same argument as Remark 2.11.(1) and 2.11.(2) of \cite{EKQR}. We then have $\sigma(X)\subseteq M_S(X\otimes S)\subseteq M_{A\otimes S}(X\otimes S)$.

        (2) A coaction on $(X,A)$ in our sense is a coaction on the Hilbert $A$-module $X$ in the sense of \cite[Definition~2.2]{BS}.
    \end{rmks}

    \begin{rmk}\label{Unify.act.coact.}\rm
        Let $G$ be a locally compact group and $(X,A)$ be a $C^*$-cor\-re\-spond\-ence. Theorem \ref{actions=coactions} says that every action of $G$ on $(X,A)$ in the sense of \cite[Definition 2.5]{EKQR} determines a coaction of the Hopf $C^*$-algebra $C_0(G)$ on $(X,A)$, and one can define in this way a one-to-one correspondence between actions of $G$ on $(X,A)$ and coactions of $C_0(G)$ on $(X,A)$. On the other hand, a nondegenerate coaction of $G$ \cite[Definition 2.10]{EKQR} is by definition a coaction $(\sigma,\delta)$ of the Hopf $C^*$-algebra $C^*(G)$ on $(X,A)$ such that $\delta$ is injective. Definition \ref{DefofCoactions} thus unifies the notions of actions and nondegenerate coactions of locally compact groups on $C^*$-cor\-re\-spond\-ences.
    \end{rmk}

    By Proposition 2.27 (Proposition 2.30, respectively) of \cite{EKQR}, an action (nondegenrate coaction, respectively) of a locally compact group $G$ on $(X,A)$ determines an action (coaction, respectively) of $G$ on $\mathcal{K}(X)$, and the left action $\varphi_A$ satisfies an equivariance condition. The next proposition generalizes this in the Hopf $C^*$-algebra setting. Recall that we identify $\mathcal{K}(X_1\otimes X_2)=\mathcal{K}(X_1)\otimes\mathcal{K}(X_2)$ for two Hilbert modules $X_1$ and $X_2$. In particular, if $(X,A)$ is a $C^*$-cor\-re\-spond\-ence and $C$ is a $C^*$-algebra then $\mathcal{K}(X\otimes C)=\mathcal{K}(X)\otimes C$.

    \begin{prop}\label{coaction on KX}
        Let $(\sigma,\delta)$ be a coaction of a Hopf $C^*$-algebra $S$ on a $C^*$-cor\-re\-spond\-ence $(X,A)$. Then the nondegenerate homomorphism
            \[\sigma^{(1)}:\mathcal{K}(X)\rightarrow M(\mathcal{K}(X\otimes S))= M(\mathcal{K}(X)\otimes S)\]
        is a coaction of $S$ on $\mathcal{K}(X)$ and the left action $\varphi_A$ is $\delta$-$\sigma^{(1)}$ equivariant, that is, $\overline{\varphi_A\otimes{\rm id}_S}\circ\delta=\overline{\sigma^{(1)}}\circ\varphi_A$. If $\delta$ is injective then so is $\sigma^{(1)}$.
    \end{prop}

    \begin{proof}
        As noted in \cite[Remark~2.8.(a)]{BS0}, $\sigma^{(1)}$ satisfies the coaction identity. We also have
        \begin{equation}\label{Sec.3.Eqn.for.Coact.KX}
        \begin{aligned}
            \overline{\sigma^{(1)}(\mathcal{K}(X))(1_{M(\mathcal{K}(X))}\otimes S)}
            &= \overline{\sigma(X)\sigma(X)^*\varphi_{M(A\otimes S)}(1_{M(A)}\otimes S)} \\
            &= \overline{\sigma(X)\big(\varphi_{M(A\otimes S)}(1_{M(A)}\otimes S)\,\sigma(X)\big)^*} \\
            &= \overline{\sigma(X)\big((\sigma(X)\cdot(1_{M(A)}\otimes S))\cdot(1_{M(A)}\otimes S)\big)^*} \\
            &= \overline{\big(\sigma(X)\cdot(1_{M(A)}\otimes S)\big)\big(\sigma(X)\cdot(1_{M(A)}\otimes S)\big)^*} \\
            &= \overline{(X\otimes S)(X\otimes S)^*} = \mathcal{K}(X\otimes S),
        \end{aligned}
        \end{equation}
        in the third and fifth step of which we use the coaction nondegeneracy of $\sigma$. This shows that $\sigma^{(1)}$ satisfies the coaction nondegeneracy, and thus $\sigma^{(1)}$ is a coaction.

        The first relation of \eqref{Pre.psi.1} and the fact that $(\sigma,\delta)$ is a correspondence homomorphism yield
            \[\overline{\sigma^{(1)}}(\varphi_A(a))\,\sigma(\xi)=\sigma(\varphi_A(a)\xi)=\varphi_{M(A\otimes S)}(\delta(a))\,\sigma(\xi).\]
        for $a\in A$ and $\xi\in X$. Multiplying by $1_{M(A)}\otimes s$ on both end sides from the right gives
            \[\overline{\sigma^{(1)}}(\varphi_A(a))\big(\sigma(\xi)\cdot(1_{M(A)}\otimes s)\big)=\varphi_{M(A\otimes S)}(\delta(a))\big(\sigma(\xi)\cdot(1_{M(A)}\otimes s)\big)\]
        which leads to $\overline{\sigma^{(1)}}(\varphi_A(a))=\varphi_{M(A\otimes S)}(\delta(a))$ by the coaction nondegeneracy of $\sigma$. But $\varphi_{M(A\otimes S)}=\overline{\varphi_A\otimes{\rm id}_S}$ by definition, and then the $\delta$-$\sigma^{(1)}$ equivariancy of $\varphi_A$ follows. For the last assertion, see the comment below \cite[Lemma~2.4]{Kat}.
    \end{proof}

    \begin{defn}\label{Sec.3.delta.invariant}\rm
        Let $(\sigma,\delta)$ be a coaction of a Hopf $C^*$-algebra $S$ on a $C^*$-cor\-re\-spond\-ence $(X,A)$. We say that the ideal $J_X$ is \emph{weakly $\delta$-invariant} if
            \[\delta(J_X)(1_{M(A)}\otimes S)\subseteq J_X\otimes S.\]
    \end{defn}

    \begin{rmk}\label{WI.for.delta}\rm
        The coaction nondegeneracy of $\delta$ implies that $J_X$ is weakly $\delta$-invariant if and only if $\delta(J_X)(A\otimes S)\subseteq J_X\otimes S$, namely
            \[\delta(J_X)\subseteq M(A\otimes S;J_X\otimes S).\]
    \end{rmk}

    Under the assumption of the last inclusion in Remark \ref{WI.for.delta} with $S=C^*(G)$, it was proved in \cite[Proposition~3.1]{KQRo2} that every coaction of a locally compact group $G$ on $(X,A)$ induces a coaction of $G$ on the associated Cuntz-Pimsner algebra $\mathcal{O}_X$. Modifying the proof of \cite[Proposition 3.1]{KQRo2} we now prove the next theorem.

    \begin{thm}\label{induced coactions on O_X}
        Let $(\sigma,\delta)$ be a coaction of a Hopf $C^*$-algebra $S$ on a $C^*$-cor\-re\-spond\-ence $(X,A)$ such that the ideal $J_X$ is weakly $\delta$-invariant. Then the representation
        \[
            (\overline{k_X\otimes{\rm id}_S}\circ\sigma,\overline{k_A\otimes{\rm id}_S}\circ\delta):(X,A)\rightarrow M_{A\otimes S}(\mathcal{O}_X\otimes S)
        \]
        is covariant, and its integrated form $\zeta:=(\overline{k_X\otimes{\rm id}_S}\circ\sigma)\times(\overline{k_A\otimes{\rm id}_S}\circ\delta)$ is a coaction of $S$ on $\mathcal{O}_X$ such that the diagram
        \begin{equation}\label{Sec.3.diagram}
        \begin{gathered}
            \xymatrix{(X,A) \ar[rr]^-{(\sigma,\delta)} \ar[d]_-{(k_X,k_A)}
            && (M_{A\otimes S}(X\otimes S),M(A\otimes S)) \ar[d]^-{(\overline{k_X\otimes{\rm id}_S},\overline{k_A\otimes{\rm id}_S})} \\
                \mathcal{O}_X \ar[rr]_-{\zeta} && M_{A\otimes S}(\mathcal{O}_X\otimes S)}
        \end{gathered}
        \end{equation}
        commutes. If $\delta$ is injective then so is $\zeta$.
    \end{thm}

    \begin{proof}
        Let us first prove that $(\overline{k_X\otimes{\rm id}_S}\circ\sigma,\overline{k_A\otimes{\rm id}_S}\circ\delta)$ is covariant, that is,
            \begin{equation*}\label{Sec.3.Lemma1.Pf.1}
                \big(\overline{k_X\otimes{\rm id}_S}\circ\sigma\big)^{(1)}\circ\varphi_A=\overline{k_A\otimes{\rm id}_S}\circ\delta
            \end{equation*}
        on $J_X$. Since $\sigma(X)\subseteq M_{A\otimes S}(X\otimes S)$ and thus $\sigma^{(1)}(\mathcal{K}(X))\subseteq\mathcal{K}(M_{A\otimes S}(X\otimes S))$, we have
            \[\big(\overline{k_X\otimes{\rm id}_S}\circ\sigma\big)^{(1)}=\overline{(k_X\otimes{\rm id}_S)^{(1)}}\circ\sigma^{(1)}\]
        on $\mathcal{K}(X)$ by Theorem \ref{DKQ,Cor.A.14}.(ii). We then have
        \begin{align*}
            \big(\overline{k_X\otimes{\rm id}_S}\circ\sigma\big)^{(1)}\circ\varphi_A
            &=\overline{(k_X\otimes{\rm id}_S)^{(1)}}\circ\sigma^{(1)}\circ\varphi_A \\
            &=\overline{(k_X\otimes{\rm id}_S)^{(1)}}\circ\overline{\varphi_{A\otimes S}}\circ\delta
        \end{align*}
        on $J_X$ since $\overline{\sigma^{(1)}}\circ\varphi_A=\overline{\varphi_{A\otimes S}}\circ\delta$ by \cite[Lemma~3.3]{KQRo1}. Hence, the requirement that $(\overline{k_X\otimes{\rm id}_S}\circ\sigma,\overline{k_A\otimes{\rm id}_S}\circ\delta)$ be covariant amounts to that
            \begin{equation*}\label{Sec.3.Lemma1.Pf.4}
                \overline{(k_X\otimes{\rm id}_S)^{(1)}}\circ\overline{\varphi_{A\otimes S}}\circ\delta=\overline{k_A\otimes{\rm id}_S}\circ\delta
            \end{equation*}
       holds on $J_X$. By Remark \ref{WI.for.delta}, this equality will follow if we show that
            \[\overline{(k_X\otimes{\rm id}_S)^{(1)}}\circ\overline{\varphi_{A\otimes S}}=\overline{k_A\otimes{\rm id}_S}\]
        on $M(A\otimes S;J_X\otimes S)$. But, this is the content of Lemma \ref{Lemma.for.Ext.Main.Thm.of.KQR2}, and therefore the representation $(\overline{k_X\otimes{\rm id}_S}\circ\sigma,\overline{k_A\otimes{\rm id}_S}\circ\delta)$ is covariant.

        We now show that $\zeta$ is a coaction of $(S,\Delta)$ on $\mathcal{O}_X$. Since
            \begin{align*}
                \overline{(1_{M(\mathcal{O}_X)}\otimes S)\zeta(k_X(X))}
                &= \overline{\overline{k_A\otimes{\rm id}_S}(1_{M(A)}\otimes S)\overline{k_X\otimes{\rm id}_S}(\sigma(X))} \\
                &= \overline{k_X\otimes{\rm id}_S\big(\varphi_{M(A\otimes S)}(1_{M(A)}\otimes S)\,\sigma(X)\big)} \\
                &= \overline{k_X(X)\odot S},
            \end{align*}
        we have
            \[
                \overline{\zeta(k_X(X)^*)(1_{M(\mathcal{O}_X)}\otimes S)}
                = \big(\,\overline{(1_{M(\mathcal{O}_X)}\otimes S)\zeta(k_X(X))}\,\big)^*
                = \overline{k_X(X)^*\odot S}.
            \]
        We also have $\overline{\zeta(k_X(X))(1_{M(\mathcal{O}_X)}\otimes S)}=\overline{k_X(X)\odot S}$. From these and the coaction nondegeneracy of $\delta$, we can deduce that $\zeta$ satisfies the coaction nondegeneracy.

        The coaction nondegeneracy of $\zeta$ implies $\zeta(\mathcal{O}_X)\subseteq M_{A\otimes S}(\mathcal{O}_X\otimes S)$, and then we have the commutative diagram \eqref{Sec.3.diagram}.

        We can easily see that $\overline{\zeta\otimes{\rm id}_S}\circ\zeta\circ k_A=\overline{{\rm id}_{\mathcal{O}_X}\otimes\Delta}\circ\zeta\circ k_A$ by \eqref{Sec.3.diagram}, strict continuity, and the coaction identity of $\delta$. To prove the corresponding equality for $k_X$, we first note the followings. Let $x\in A\otimes S$ and $m\in M_{A\otimes S}(\mathcal{O}_X\otimes S)$. Then
            \begin{align*}
                (\zeta\otimes{\rm id}_S)\big((k_A\otimes{\rm id}_S)(x)\,m\big)
                &=\overline{k_A\otimes{\rm id}_S\otimes{\rm id}_S}\big((\delta\otimes{\rm id}_S)(x)\big)\,\overline{\zeta\otimes{\rm id}_S}(m), \\
                ({\rm id}_{\mathcal{O}_X}\otimes\Delta)\big((k_A\otimes{\rm id}_S)(x)\,m\big)
                &=\overline{k_A\otimes{\rm id}_S\otimes{\rm id}_S}\big(({\rm id}_{A}\otimes\Delta)(x)\big)\,\overline{{\rm id}_{\mathcal{O}_X}\otimes\Delta}(m),
            \end{align*}
        and similarly for $(\zeta\otimes{\rm id}_S)\big(m\,(k_A\otimes{\rm id}_S)(x)\big)$ and $({\rm id}_{\mathcal{O}_X}\otimes\Delta)\big(m\,(k_A\otimes{\rm id}_S)(x)\big)$. From these relations and also the nondegeneracy of $\delta\otimes{\rm id}_S$ and ${\rm id}_A\otimes\Delta$, we deduce that the restrictions
            \[\overline{\zeta\otimes{\rm id}_S},\ \overline{{\rm id}_{\mathcal{O}_X}\otimes\Delta}:M_{A\otimes S}(\mathcal{O}_X\otimes S)\rightarrow M_{A\otimes S\otimes S}(\mathcal{O}_X\otimes S\otimes S)\]
        are $(A\otimes S)$-strict to $(A\otimes S\otimes S)$-strictly continuous (cf.\ \cite[Lemma~A.5]{DKQ}). Therefore the following compositions
            \begin{multline}\label{Sec.3.Pf.Thm.mult.1}
                \overline{\zeta\otimes{\rm id}_S}\circ\overline{k_X\otimes{\rm id}_S},\ \overline{{\rm id}_{\mathcal{O}_X}\otimes\Delta}\circ\overline{k_X\otimes{\rm id}_S}:\\
                M_{A\otimes S}(X\otimes S)\rightarrow M_{A\otimes S\otimes S}(\mathcal{O}_X\otimes S\otimes S)
            \end{multline}
        are $(A\otimes S)$-strict to $(A\otimes S\otimes S)$-strictly continuous. Similarly, both maps
            \[\overline{\sigma\otimes{\rm id}_S},\ \overline{{\rm id}_X\otimes\Delta}:M_{A\otimes S}(X\otimes S)\rightarrow M_{A\otimes S\otimes S}(X\otimes S\otimes S)\]
        are $(A\otimes S)$-strict to $(A\otimes S\otimes S)$-strictly continuous, and hence so are the maps
            \begin{multline}\label{Sec.3.Pf.Thm.mult.2}
                \overline{k_X\otimes{\rm id}_S\otimes{\rm id}_S}\circ\overline{\sigma\otimes{\rm id}_S},\ \overline{k_X\otimes{\rm id}_S\otimes{\rm id}_S}\circ\overline{{\rm id}_X\otimes\Delta}:\\
                M_{A\otimes S}(X\otimes S)\rightarrow M_{A\otimes S\otimes S}(\mathcal{O}_X\otimes S\otimes S).
            \end{multline}
        Since the equalities
            \begin{align*}
                \overline{\zeta\otimes{\rm id}_S}\circ\overline{k_X\otimes{\rm id}_S}
                &=\overline{k_X\otimes{\rm id}_S\otimes{\rm id}_S}\circ\overline{\sigma\otimes{\rm id}_S}, \\
                \overline{k_X\otimes{\rm id}_S\otimes{\rm id}_S}\circ\overline{{\rm id}_X\otimes\Delta}
                &=\overline{{\rm id}_{\mathcal{O}_X}\otimes\Delta}\circ\overline{k_X\otimes{\rm id}_S}
            \end{align*}
        hold on $X\odot S$ which is $(A\otimes S)$-strictly dense in $M_{A\otimes S}(X\otimes S)$ and since $\sigma(X)\subseteq M_{A\otimes S}(X\otimes S)$, we now have
            \begin{align*}
                \overline{\zeta\otimes{\rm id}_S}\circ\zeta\circ k_X
                &=\overline{\zeta\otimes{\rm id}_S}\circ\overline{k_X\otimes{\rm id}_S}\circ\sigma \\
                &=\overline{k_X\otimes{\rm id}_S\otimes{\rm id}_S}\circ\overline{\sigma\otimes{\rm id}_S}\circ\sigma \\
                &=\overline{k_X\otimes{\rm id}_S\otimes{\rm id}_S}\circ\overline{{\rm id}_X\otimes\Delta}\circ\sigma \\
                &=\overline{{\rm id}_{\mathcal{O}_X}\otimes\Delta}\circ\overline{k_X\otimes{\rm id}_S}\circ\sigma
                =\overline{{\rm id}_{\mathcal{O}_X}\otimes\Delta}\circ\zeta\circ k_X
            \end{align*}
        by the $(A\otimes S)$-strict to $(A\otimes S\otimes S)$-strict continuity of the maps of \eqref{Sec.3.Pf.Thm.mult.1} and \eqref{Sec.3.Pf.Thm.mult.2} and also by the coaction identity of $\sigma$. Thus $\zeta$ satisfies the coaction identity.

        For the last assertion of the theorem, assume that $\delta$ is injective. We only need to show by \cite[Theorem~6.4]{Kat} that the injective covariant representation $(\overline{k_X\otimes{\rm id}_S}\circ\sigma,\overline{k_A\otimes{\rm id}_S}\circ\delta)$ admits a gauge action. Let $\beta:\mathbb{T}\rightarrow Aut(\mathcal{O}_X)$ be the gauge action. Note that for each $z\in\mathbb{T}$, the strict extension $\overline{\beta_z\otimes{\rm id}_S}$ on $M(\mathcal{O}_X\otimes S)$ maps $M_{A\otimes S}(\mathcal{O}_X\otimes S)$ onto itself. Then the composition
            \begin{multline}\label{Sec.3.Comp.Str.Maps}
                (\overline{\beta_z\otimes{\rm id}_S}\circ\overline{k_X\otimes{\rm id}_S},\,\overline{\beta_z\otimes{\rm id}_S}\circ\overline{k_A\otimes{\rm id}_S}): \\
                (M_{A\otimes S}(X\otimes S),M(A\otimes S))\rightarrow M_{A\otimes S}(\mathcal{O}_X\otimes S)
            \end{multline}
        gives a representation which is clearly $(A\otimes S)$-strictly continuous. Since the equalities
            \begin{align*}
                \overline{\beta_z\otimes{\rm id}_S}\circ\overline{k_X\otimes{\rm id}_S}(m)
                &=z\,\overline{k_X\otimes{\rm id}_S}(m), \\
                \quad\overline{\beta_z\otimes{\rm id}_S}\circ\overline{k_A\otimes{\rm id}_S}(n)
                &=\overline{k_A\otimes{\rm id}_S}(n)
            \end{align*}
        are valid for $m\in X\odot S$ and $n\in A\odot S$, and the representation \eqref{Sec.3.Comp.Str.Maps} is $(A\otimes S)$-strictly continuous, the above equalities still hold for $m\in M_{A\otimes S}(X\otimes S)$ and $n\in M(A\otimes S)$. Since $\sigma(X)\subseteq M_{A\otimes S}(X\otimes S)$, it thus follows that
            \[\overline{\beta_z\otimes{\rm id}_S}\circ\overline{k_X\otimes{\rm id}_S}\circ\sigma=z\,\overline{k_X\otimes{\rm id}_S}\circ\sigma,\]
        and similarly that $\overline{\beta_z\otimes{\rm id}_S}\circ\overline{k_A\otimes{\rm id}_S}\circ\delta=\overline{k_A\otimes{\rm id}_S}\circ\delta$. This proves that the restrictions of $\overline{\beta_z\otimes{\rm id}_S}$ to $\zeta(\mathcal{O}_X)$ $(z\in\mathbb{T})$ define a gauge action of $\mathbb{T}$ on $\zeta(\mathcal{O}_X)$, which establishes the theorem.
    \end{proof}

    \begin{defn}\rm
        We call $\zeta$ in Theorem \ref{induced coactions on O_X} the coaction \emph{induced} by $(\sigma,\delta)$.
    \end{defn}

    \begin{rmks}\label{Sec.3.Rmk.to.Thm}\rm
        (1) Let $G$ be a locally compact group. If $(\sigma,\delta)$ is a coaction of $C_0(G)$ on $(X,A)$, then $\overline{\delta(J_X)(1_{M(A)}\otimes S)}=J_X\otimes S$ by \cite[Lemma 2.6.(a)]{HaoNg} and Theorem \ref{actions=coactions}. Hence, $J_X$ is automatically weakly $\delta$-invariant in this case.

        (2) Replacing in the diagram \eqref{Sec.3.diagram} the $(A\otimes S)$-multiplier correspondence and $(A\otimes S)$-multiplier algebra by $(M_S(X\otimes S),M_S(A\otimes S))$ and $M_S(\mathcal{O}_X\otimes S)$, respectively, we can regard $(\overline{k_X\otimes{\rm id}_S},\overline{k_A\otimes{\rm id}_S})$ as the $S$-strict extension by Remarks \ref{Prel.Str.Ext.}.(2).
    \end{rmks}

\section{Reduced crossed product correspondences}\label{Sec.4}

    In this section and the next, we restrict our attention to coactions of reduced Hopf $C^*$-algebras defined by well-behaved multiplicative unitaries.

    \begin{notation}\rm
        To simplify the notations, we often write $S$ and $\widehat{S}$, respectively, for the ``reduced'' and ``dual reduced'' Hopf $C^*$-algebras $S_V$ and $\widehat{S}_V$ defined by a well-behaved multiplicative unitary $V$. We also write $\mathscr{K}$ for $\mathcal{K}(\mathcal{H})$. 

        Let $(\sigma,\delta)$ be a coaction of $S$ on $(X,A)$ and $\iota_S:S\hookrightarrow M(\mathscr{K})$ be the inclusion map. As $\delta_\iota$ in \eqref{Prel.delta.iota}, we denote by $\sigma_\iota$ the composition
            \[\sigma_\iota=\overline{{\rm id}_X\otimes\iota_{S}}\circ\sigma,\]
        where $\overline{{\rm id}_X\otimes\iota_{S}}$ is the strict extension. Evidently, $(\sigma_\iota,\delta_\iota)$ is a nondegenerate correspondence homomorphism:
            \[ \xymatrix  @C=0pc {(X,A) \ar[rr]^-{(\sigma_\iota,\delta_\iota)} \ar[dr]_-{(\sigma,\delta)}
            && (M(X\otimes\mathscr{K}),M(A\otimes\mathscr{K})) \\
            & (M(X\otimes S),M(A\otimes S)) \ar[ur]_-{\qquad(\overline{{\rm id}_X\otimes\iota_S},\,\overline{{\rm id}_A\otimes\iota_S})} &
            }\]

    \end{notation}

    For a coaction $(\sigma,\delta)$ of $S$ on $(X,A)$, we have $\sigma_\iota(X)\cdot(1_{M(A)}\otimes\widehat{S})\subseteq M(X\otimes\mathscr{K})$ and similarly for $\varphi_{M(A\otimes\mathscr{K})}(1_{M(A)}\otimes\widehat{S})\,\sigma_\iota(X)$ since $(M(X\otimes\mathscr{K}),M(A\otimes\mathscr{K}))$ is a $C^*$-cor\-re\-spond\-ence.

    \begin{lem}[{\cite[Proposition~1.3]{Bui}}]\label{Lemma similar to BS's}
        The norm closures in $M(X\otimes\mathscr{K})$ of the subspaces $\sigma_\iota(X)\cdot(1_{M(A)}\otimes\widehat{S})$ and $\varphi_{M(A\otimes\mathscr{K})}(1_{M(A)}\otimes\widehat{S})\,\sigma_\iota(X)$ coincide.
    \end{lem}

    Although \cite[Proposition~1.3]{Bui} assumes the regularity condition \cite{BS} for multiplicative unitaries, the proof uses only the pentagonal relation, and hence works the same for well-behaved multiplicative unitaries.

    We denote by $X\rtimes_\sigma\widehat{S}$ the norm closure of the subspaces considered in Lemma~\ref{Lemma similar to BS's}:
        \[X\rtimes_\sigma\widehat{S}:=\overline{\sigma_\iota(X)\cdot(1_{M(A)}\otimes\widehat{S})}
        =\overline{\varphi_{M(A\otimes\mathscr{K})}(1_{M(A)}\otimes\widehat{S})\,\sigma_\iota(X)}.\]
    It is obvious that $X\rtimes_\sigma\widehat{S}$ is a Hilbert $(A\rtimes_\delta\widehat{S})$-module with respect to the restriction of the operations on $(M(X\otimes\mathscr{K}),M(A\otimes\mathscr{K}))$.

    \begin{thm}\label{crossed product correspondences}
        Let $(\sigma,\delta)$ be a coaction of a reduced Hopf $C^*$-algebra $S$ on a $C^*$-cor\-re\-spon\-dence $(X,A)$. Then $(X\rtimes_\sigma\widehat{S},A\rtimes_\delta\widehat{S})$ is a nondegenerate $C^*$-cor\-re\-spond\-ence such that the inclusion
            \begin{equation*}\label{Sec.4.Inc.in.Thm}
                (X\rtimes_\sigma\widehat{S},A\rtimes_\delta\widehat{S})\hookrightarrow
                (M(X\otimes\mathscr{K}),M(A\otimes\mathscr{K}))
            \end{equation*}
        is a nondegenerate correspondence homomorphism. The left action $\varphi_{A\rtimes_\delta\widehat{S}}$ is injective if $\varphi_A$ is injective. Also,
            \begin{equation*}\label{Sec.4.Cpt.Eq}
                \mathcal{K}(X\rtimes_\sigma\widehat{S})=\mathcal{K}(X)\rtimes_{\sigma^{(1)}}\widehat{S},
            \end{equation*}
        where $\sigma^{(1)}$ is the coaction in Proposition \ref{coaction on KX}, and
            \begin{equation*}\label{PhiDelta=DeltaMu}
                \varphi_{A\rtimes_\delta\widehat{S}}\big(\delta_\iota(a)(1_{M(A)}\otimes x)\big)=\overline{\sigma^{(1)}_\iota}(\varphi_A(a))(1_{M(\mathcal{K}(X))}\otimes x)
            \end{equation*}
        for $a\in A$ and $x\in\widehat{S}$.
    \end{thm}

    \begin{proof}
        The first assertion amounts to saying that
        \begin{itemize}
            \item[(i)] the Hilbert $(A\rtimes_\delta\widehat{S})$-module $X\rtimes_\sigma\widehat{S}$ is a nondegenerate $C^*$-cor\-re\-spond\-ence such that $\varphi_{A\rtimes_\delta\widehat{S}}=\varphi_{M(A\otimes\mathscr{K})}|_{A\rtimes_\delta\widehat{S}}$, namely
                    \[\overline{\varphi_{M(A\otimes\mathscr{K})}\big(\delta_\iota(A)(1_{M(A)}\otimes\widehat{S})\big)\,
                    \sigma_\iota(X)\cdot(1_{M(A)}\otimes\widehat{S})}=X\rtimes_\sigma\widehat{S};\]
            \item[(ii)] the inclusion $(X\rtimes_\sigma\widehat{S},A\rtimes_\delta\widehat{S})\hookrightarrow
                (M(X\otimes\mathscr{K}),M(A\otimes\mathscr{K}))$ is a nondegenerate correspondence homomorphism, namely
                    \[\overline{(X\rtimes_\sigma\widehat{S})\cdot(A\otimes\mathscr{K})}=X\otimes\mathscr{K},\quad
                    \overline{(A\rtimes_\delta\widehat{S})(A\otimes\mathscr{K})}=A\otimes\mathscr{K}.\]
        \end{itemize}
        Lemma \ref{Lemma similar to BS's} shows that
            \[\overline{\varphi_{M(A\otimes\mathscr{K})}(1_{M(A)}\otimes\widehat{S})\,
                \sigma_\iota(X)\cdot(1_{M(A)}\otimes\widehat{S})}=\overline{\sigma_\iota(X)\cdot(1_{M(A)}\otimes\widehat{S})}.\]
        Since $\varphi_A$ is nondegenerate, this equality combined with the following
            \[\varphi_{M(A\otimes\mathscr{K})}(\delta_\iota(A))\,\sigma_\iota(X)=\sigma_\iota(\varphi_A(A)X)\]
        gives (i). Since $S$ and $\widehat{S}$ are both nondegenerate subalgebras of $M(\mathscr{K})$, we have
            \begin{equation}\label{Sec.4.in.the.Pf.thm}
            \begin{aligned}
                \overline{(X\rtimes_\sigma\widehat{S})\cdot(A\otimes\mathscr{K})}
                &=\overline{\sigma_\iota(X)\cdot(A\otimes \widehat{S}\mathscr{K})} \\
                &=\overline{\sigma_\iota(X)\cdot(1_{M(A)}\otimes S)\cdot(A\otimes \widehat{S}\mathscr{K})} \\
                &=\overline{(X\otimes S)\cdot(A\otimes\mathscr{K})} = X\otimes\mathscr{K}
            \end{aligned}
            \end{equation}
        and similarly $\overline{(A\rtimes_\delta\widehat{S})(A\otimes\mathscr{K})}=A\otimes\mathscr{K}$. This verifies (ii), and the first assertion of the theorem is established. Since $\varphi_{A\rtimes_\delta\widehat{S}}$ is the restriction of $\overline{\varphi_A\otimes{\rm id}_{\mathscr{K}}}$ which is injective if $\varphi_A$ is, the assertion on the injectivity of $\varphi_{A\rtimes_\delta\widehat{S}}$ follows.

        As in the computation \eqref{Sec.3.Eqn.for.Coact.KX}, but using Lemma \ref{Lemma similar to BS's} instead of coaction nondegeneracy, we can deduce the equality $\mathcal{K}(X\rtimes_\sigma\widehat{S})=\mathcal{K}(X)\rtimes_{\sigma^{(1)}}\widehat{S}$. Finally, 
        \begin{align*}
            \varphi_{A\rtimes_\delta\widehat{S}}\big(\delta_\iota(a)(1_{M(A)}\otimes x)\big)
            &= \overline{\varphi_A\otimes{\rm id}_{\mathscr{K}}}\circ\overline{{\rm id}_A\otimes\iota_S}(\delta(a))\,(1_{M(\mathcal{K}(X))}\otimes x)  \\
            &= \overline{{\rm id}_{\mathcal{K}(X)}\otimes\iota_{S}}\circ\overline{\varphi_A\otimes{\rm id}_S}(\delta(a))\,(1_{M(\mathcal{K}(X))}\otimes x) \\
            &= \overline{{\rm id}_{\mathcal{K}(X)}\otimes\iota_{S}}\circ\overline{\sigma^{(1)}}(\varphi_A(a))\,(1_{M(\mathcal{K}(X))}\otimes x) \\
            &= \overline{\sigma^{(1)}_\iota}(\varphi_A(a))(1_{M(\mathcal{K}(X))}\otimes x),
        \end{align*}
        in the third step of which we use the $\delta$-$\sigma^{(1)}$ equivariancy of $\varphi_A$ obtained in Proposition \ref{coaction on KX}. This completes the proof.
    \end{proof}

    \begin{defn}\label{Sec.4.Def}\rm
        We call the $C^*$-correspondence $(X\rtimes_\sigma\widehat{S},A\rtimes_\delta\widehat{S})$ in Theorem~\ref{crossed product correspondences} the \emph{reduced crossed product correspondence} of $(X,A)$ by the coaction $(\sigma,\delta)$ of $S$.
    \end{defn}

    \begin{rmk}\label{sigmaL,deltaL define a non-deg.corr.hom.}\rm
        We require no universal property of the crossed product $A\rtimes_\delta\widehat{S}$ to define the left action $\varphi_{A\rtimes_\delta\widehat{S}}:A\rtimes_\delta\widehat{S}\rightarrow\mathcal{L}(X\rtimes_\sigma\widehat{S})$. It is just the restriction of $\varphi_{M(A\otimes\mathscr{K})}$.
    \end{rmk}

    \begin{rmk}\label{Sec.4.Rmk.for.Justi.}\rm
        For an action $(\gamma,\alpha)$ of a locally compact group $G$ on $(X,A)$, one can form the crossed product correspondence $(X\rtimes_{\gamma,r}G,A\rtimes_{\alpha,r}G)$ by \cite[Proposition~3.2]{EKQR}. We will see in Corollary \ref{Appendix.B.Cor.1} that it is isomorphic to the reduced crossed product correspondence $(X\rtimes_{\sigma^\gamma_G}\widehat{S}_{\widehat{W}_G},A\rtimes_{\delta^\alpha_G}\widehat{S}_{\widehat{W}_G})$, where $(\sigma^\gamma_G,\delta^\alpha_G)$ is the coaction of the Hopf $C^*$-algebra $S_{\widehat{W}_G}$ given in \eqref{Appendix.B.Cor.1.Eqn}. On the other hand, if $(\sigma,\delta)$ is a nondegenerate coaction of $G$ on $(X,A)$ (\cite[Definition 2.10]{EKQR}) and if $\sigma_\lambda:=\overline{{\rm id}_X\otimes\lambda}\circ\sigma$ as \eqref{Prel.coaction.of.G}, then the crossed product correspondence by $(\sigma,\delta)$ in the sense of \cite[Proposition~3.9]{EKQR} is just the reduced crossed product correspondence by the coaction $(\sigma_\lambda,\delta_\lambda)$ of the Hopf $C^*$-algebra $S_{W_G}$. Construction in Theorem~\ref{crossed product correspondences} thus extends both of the crossed product correspondences by actions and nondegenerate coactions of locally compact groups on $C^*$-cor\-re\-spond\-ences.
    \end{rmk}

    As in \cite[Remark 2.7]{KQRo2}, we have the following corollary, the proof of which is routine.

    \begin{cor}\label{Sec.4.jxja}
        Let $(\sigma,\delta)$ be a coaction of $S$ on $(X,A)$. Then the map
            \[(j_X^\sigma,j_A^\delta):(X,A)\rightarrow (M(X\rtimes_\sigma\widehat{S}),M(A\rtimes_\delta\widehat{S}))\]
        defined by
            \[j_X^\sigma(\xi)\cdot c:=\sigma_\iota(\xi)\cdot c,\quad j_A^\delta(a)c:=\delta_\iota(a)c\]
        for $\xi\in X$, $a\in A$, and $c\in A\rtimes_\delta\widehat{S}$ is a nondegenerate correspondence homomorphism such that $j_X^\sigma(X)\subseteq M_{A\rtimes_\delta\widehat{S}}(X\rtimes_\sigma\widehat{S})$.
    \end{cor}

    \begin{rmk}\label{Sec.4.Rmk.for.Pf.Main.Thm}\rm
        Since $(X\rtimes_\sigma\widehat{S},A\rtimes_\delta\widehat{S})$ is by Theorem~\ref{crossed product correspondences} a nondegenerate $C^*$-cor\-re\-spond\-ence, we have by Theorem~\ref{DKQ,Cor.A.14}.(i) the $(A\rtimes_\delta\widehat{S})$-strict extension
            \[(\overline{k_{X\rtimes_\sigma\widehat{S}}},\overline{k_{A\rtimes_\delta\widehat{S}}}):
            (M_{A\rtimes_\delta\widehat{S}}(X\rtimes_\sigma\widehat{S}),M(A\rtimes_\delta\widehat{S}))\hookrightarrow M_{A\rtimes_\delta\widehat{S}}(\mathcal{O}_{X\rtimes_\sigma\widehat{S}})\]
        for the universal covariant representation $(k_{X\rtimes_\sigma\widehat{S}},k_{A\rtimes_\delta\widehat{S}})$. The composition
            \[(\overline{k_{X\rtimes_\sigma\widehat{S}}}\circ j_X^\sigma,\overline{k_{A\rtimes_\delta\widehat{S}}}\circ j_A^\delta):(X,A)\rightarrow M_{A\rtimes_\delta\widehat{S}}(\mathcal{O}_{X\rtimes_\sigma\widehat{S}}).\]
        then gives a representation of $(X,A)$. It can be shown that this representation is covariant although we do not need this fact in the sequel.
    \end{rmk}

\section{Reduced crossed products}\label{Sec.5}

    In this section, we first show that the $C^*$-cor\-re\-spond\-ence $(X\rtimes_\sigma\widehat{S},A\rtimes_\delta\widehat{S})$ has a representation $(k_X\rtimes_\sigma{\rm id},k_A\rtimes_\delta{\rm id})$ on the reduced crossed product $\mathcal{O}_X\rtimes_\zeta\widehat{S}$. We then provide a couple of equivalent conditions that this representation is covariant, which is readily seen to be the case if the ideal $J_{X\rtimes_\sigma\widehat{S}}$ of $A\rtimes_\delta\widehat{S}$ is generated by $\delta_\iota(J_X)$ or the left action $\varphi_A$ is injective. Under this covariance condition, the integrated form of the representation $(X\rtimes_\sigma\widehat{S},A\rtimes_\delta\widehat{S})$ gives an isomorphism between the $C^*$-algebra $\mathcal{O}_X\rtimes_\zeta\widehat{S}$ and the Cuntz-Pimsner algebra $\mathcal{O}_{X\rtimes_\sigma\widehat{S}}$. The representation
    \begin{equation*}\label{Sec.5.Begin}
        (\overline{k_X\otimes{\rm id}_{\mathscr{K}}},\overline{k_A\otimes{\rm id}_{\mathscr{K}}}):
        (M_{A\otimes\mathscr{K}}(X\otimes\mathscr{K}),M(A\otimes\mathscr{K}))\rightarrow M_{A\otimes\mathscr{K}}(\mathcal{O}_X\otimes\mathscr{K})
    \end{equation*}
    will play an important role in our analysis.

    Recall that $\overline{k_X\otimes\id_C}$ denotes the $(A\otimes C)$-strict extension to $M_{A\otimes C}(X\otimes C)$.

    \begin{lem}\label{the diagram}
        Let $(X,A)$ be a $C^*$-correspondence. Let $S$ be a reduced Hopf $C^*$-algebra and $\iota_{S}:S\hookrightarrow M(\mathscr{K})$ be the inclusion. Then the following diagram commutes:
        \begin{equation}\label{Sec.5.diagram}
        \begin{gathered}
        \xymatrix{(M_{A\otimes S}(X\otimes S),M(A\otimes S))
            \ar[rr]^-{(\overline{{\rm id}_X\otimes\iota_{S}},\overline{{\rm id}_A\otimes\iota_{S}})}
            \ar[d]_-{(\overline{k_X\otimes{\rm id}_{S}},\overline{k_A\otimes{\rm id}_{S}})}
            && (M_{A\otimes\mathscr{K}}(X\otimes\mathscr{K}),M(A\otimes\mathscr{K}))
            \ar[d]^-{(\overline{k_X\otimes{\rm id}_{\mathscr{K}}},\overline{k_A\otimes{\rm id}_{\mathscr{K}}})} \\
            M_{A\otimes S}(\mathcal{O}_X\otimes S)
            \ar[rr]_-{\overline{{\rm id}_{\mathcal{O}_X}\otimes\iota_{S}}}
            && M_{A\otimes\mathscr{K}}(\mathcal{O}_X\otimes\mathscr{K}) }
        \end{gathered}
        \end{equation}
    \end{lem}

    \begin{proof}
        By \cite[Proposition A.11]{DKQ}, we see that the upper and lower horizontal maps are $(A\otimes S)$-strict to $(A\otimes\mathscr{K})$-strictly continuous. Hence the two compositions in \eqref{Sec.5.diagram} are $(A\otimes S)$-strict to $(A\otimes\mathscr{K})$-strictly continuous. Since the diagram commutes on $(X\odot S,A\odot S)$, the conclusion follows by strict continuity.
    \end{proof}

    \begin{cor}\label{Sec.5.cor.to.diagram}
        Let $(\sigma,\delta)$ be a coaction of $S$ on $(X,A)$ such that $J_X$ is weakly $\delta$-invariant. Then $\sigma_\iota(X)\subseteq M_{A\otimes\mathscr{K}}(X\otimes\mathscr{K})$ and
            \begin{equation*}\label{Sec.5.upper.zeta.iota.2}
                X\rtimes_\sigma\widehat{S}\subseteq M_{A\otimes\mathscr{K}}(X\otimes\mathscr{K}).
            \end{equation*}
        Also,
        \begin{equation}\label{Sec.5.zeta.iota}
            \overline{k_X\otimes{\rm id}_{\mathscr{K}}}(\sigma_\iota(\xi))=\zeta_\iota(k_X(\xi)), \quad
            \overline{k_A\otimes{\rm id}_{\mathscr{K}}}(\delta_\iota(a))=\zeta_\iota(k_A(a))
        \end{equation}
        for $\xi\in X$ and $a\in A$.
    \end{cor}

    \begin{proof}
        By Theorem \ref{induced coactions on O_X}, we can consider the induced coaction $\zeta$ on $\mathcal{O}_X$ making the diagram \eqref{Sec.3.diagram} commute. Combining \eqref{Sec.3.diagram} and \eqref{Sec.5.diagram} we see that
            \[\sigma_\iota(X)=\overline{{\rm id}_X\otimes\iota_S}(\sigma(X))\subseteq M_{A\otimes\mathscr{K}}(X\otimes\mathscr{K}),\]
        and thus
            \begin{align*}
            X\rtimes_\sigma\widehat{S}&=\overline{\sigma_\iota(X)\cdot(1_{M(A)}\otimes\widehat{S})}\\
            &\subseteq M_{A\otimes\mathscr{K}}(X\otimes\mathscr{K})\cdot M(A\otimes\mathscr{K})=M_{A\otimes\mathscr{K}}(X\otimes\mathscr{K}).
            \end{align*}
        The equalities of \eqref{Sec.5.zeta.iota} are also immediate from \eqref{Sec.3.diagram} and \eqref{Sec.5.diagram}.
    \end{proof}

    \begin{rmk}\label{Sec.5.Rmk.to.Cor.to.Diag}\rm
        From Corollary \ref{Sec.5.cor.to.diagram} (and also from Theorem \ref{crossed product correspondences}), we have an injective correspondence homomorphism
            \[(X\rtimes_\sigma\widehat{S},A\rtimes_\delta\widehat{S})\hookrightarrow
                (M_{A\otimes\mathscr{K}}(X\otimes\mathscr{K}),M(A\otimes\mathscr{K})).\]
        We then have
            \[\mathcal{K}(X\rtimes_\sigma\widehat{S})\subseteq \mathcal{K}\big(M_{A\otimes\mathscr{K}}(X\otimes\mathscr{K})\big) \subseteq M_{A\otimes\mathscr{K}}\big(\mathcal{K}(X\otimes\mathscr{K})\big).\]
    \end{rmk}

    \begin{prop}\label{embedding XxShat to OxShat is a Toep.rep.}
        Let $(\sigma,\delta)$ be a coaction of $S$ on $(X,A)$ such that $J_X$ is weakly $\delta$-invariant. Then, the restriction of $(\overline{k_X\otimes{\rm id}_{\mathscr{K}}},\overline{k_A\otimes{\rm id}_{\mathscr{K}}})$ to $(X\rtimes_\sigma\widehat{S},A\rtimes_\delta\widehat{S})$ defines an injective representation
            \begin{equation*}\label{Sec.5.Prop.1}
                (k_X\rtimes_\sigma{\rm id}_{\widehat{S}},k_A\rtimes_\delta{\rm id}_{\widehat{S}}):(X\rtimes_\sigma\widehat{S},A\rtimes_\delta\widehat{S})\rightarrow\mathcal{O}_X\rtimes_{\zeta}\widehat{S}
            \end{equation*}
        such that
            \begin{equation}\label{Sec.5.Rep.1}
            \begin{aligned}
                k_X\rtimes_\sigma{\rm id}_{\widehat{S}}\big(\sigma_\iota(\xi)\cdot(1_{M(A)}\otimes x)\big) &=\zeta_\iota(k_X(\xi))(1_{M(\mathcal{O}_X)}\otimes x), \\
                k_A\rtimes_\delta{\rm id}_{\widehat{S}}\big(\delta_\iota(a)(1_{M(A)}\otimes x)\big)
                &=\zeta_\iota(k_A(a))(1_{M(\mathcal{O}_X)}\otimes x), \\
                k_X\rtimes_\sigma{\rm id}_{\widehat{S}}\big(\varphi_{M(A\otimes\mathscr{K})}(1_{M(A)}\otimes x)\sigma_\iota(\xi)\big)
                &= (1_{M(\mathcal{O}_X)}\otimes x)\zeta_\iota(k_X(\xi))
            \end{aligned}
            \end{equation}
        for $\xi\in X$, $x\in\widehat{S}$, and $a\in A$.
    \end{prop}

    \begin{proof}
        Since $X\rtimes_\sigma\widehat{S}\subseteq M_{A\otimes\mathscr{K}}(X\otimes\mathscr{K})$, the restriction
            \[(k_X\rtimes_\sigma{\rm id}_{\widehat{S}},k_A\rtimes_\delta{\rm id}_{\widehat{S}}):=(\overline{k_X\otimes{\rm id}_{\mathscr{K}}}|_{X\rtimes_\sigma\widehat{S}},\,\overline{k_A\otimes{\rm id}_{\mathscr{K}}}|_{A\rtimes_\delta\widehat{S}})\]
        makes sense and is an injective representation of $(X\rtimes_\sigma\widehat{S},A\rtimes_\delta\widehat{S})$ on $M_{A\otimes \mathscr{K}}(\mathcal{O}_X\otimes\mathscr{K})$. Using the equalities \eqref{Sec.5.zeta.iota}, we have
        \begin{align*}
            k_X\rtimes_\sigma{\rm id}_{\widehat{S}}\big(\sigma_\iota(\xi)\cdot(1_{M(A)}\otimes x)\big)
            &=\overline{k_X\otimes{\rm id}_{\mathscr{K}}}\big(\sigma_\iota(\xi)\cdot(1_{M(A)}\otimes x)\big) \\
            &=\overline{k_X\otimes{\rm id}_{\mathscr{K}}}(\sigma_\iota(\xi))\,\overline{k_A\otimes{\rm id}_{\mathscr{K}}}(1_{M(A)}\otimes x) \\
            &=\zeta_\iota(k_X(\xi))(1_{M(\mathcal{O}_X)}\otimes x)
        \end{align*}
         for $\xi\in X$ and $x\in\widehat{S}$, and similarly for $k_A\rtimes_\delta{\rm id}_{\widehat{S}}$. This proves the first two equalities of \eqref{Sec.5.Rep.1}, and hence $(k_X\rtimes_\sigma{\rm id}_{\widehat{S}},k_A\rtimes_\delta{\rm id}_{\widehat{S}})$ is a representation on $\mathcal{O}_X\rtimes_\zeta\widehat{S}$. The last of \eqref{Sec.5.Rep.1} can be seen similarly.
    \end{proof}

    For an action $(\gamma,\alpha)$ of a locally compact group group $G$ on $(X,A)$, the ideal $J_{X\rtimes_{\gamma,r}G}$ for the crossed product correspondence $(X\rtimes_{\gamma,r}G,A\rtimes_{\delta,r}G)$ is known to be equal to the crossed product $J_X\rtimes_{\alpha,r}G$ if $G$ is amenable (\cite[Proposition~2.7]{HaoNg}) or if $G$ is discrete such that it is exact or $\alpha$ has Exel's Approximation Property (\cite[Theorem 5.5]{BKQR}). We now give a partial analogue of this fact in the Hopf $C^*$-algebra setting.

    \begin{prop}\label{Ext.Main.Thm.of.KQR2}
        Let $(\sigma,\delta)$ be a coaction of $S$ on $(X,A)$ such that $J_X$ is weakly $\delta$-invariant. Then
            \begin{equation}\label{Sec.5.Drop.CP.KQR2}
                \delta_\iota(J_X)(1_{M(A)}\otimes\widehat{S})\subseteq J_{X\rtimes_\sigma\widehat{S}}.
            \end{equation}
        In particular, if $J_X=A$ then $J_{X\rtimes_\sigma\widehat{S}}=A\rtimes_\delta\widehat{S}$.
    \end{prop}

    \begin{proof}
        The last assertion of the proposition is an immediate consequence of the first. Hence we only need to prove \eqref{Sec.5.Drop.CP.KQR2}, which will follow by \cite[Proposition~3.3]{Kat2} if we show that
            \[k_A\rtimes_\delta{\rm id}_{\widehat{S}}\big(\delta_\iota(J_X)(1_{M(A)}\otimes\widehat{S})\big)\subseteq (k_X\rtimes_\sigma{\rm id}_{\widehat{S}})^{(1)}\big(\mathcal{K}(X\rtimes_\sigma\widehat{S})\big)\]
         since the representation $(k_X\rtimes_\sigma{\rm id}_{\widehat{S}},k_A\rtimes_\delta{\rm id}_{\widehat{S}})$ is injective. Let us first note the following. By Theorem~\ref{DKQ,Cor.A.14}.(ii), we have
            \[\overline{(k_X\otimes{\rm id}_{\mathscr{K}})^{(1)}}=\overline{k_X\otimes{\rm id}_{\mathscr{K}}}^{\,(1)}\]
        on $\mathcal{K}(M_{A\otimes\mathscr{K}}(X\otimes\mathscr{K}))$. Hence
            \begin{align*}
                \overline{(k_X\otimes{\rm id}_{\mathscr{K}})^{(1)}}\big(\mathcal{K}(X\rtimes_\sigma\widehat{S})\big)
                &=\overline{k_X\otimes{\rm id}_{\mathscr{K}}}^{\,(1)}\big(\mathcal{K}(X\rtimes_\sigma\widehat{S})\big) \\
                &=(k_X\rtimes_\sigma{\rm id}_{\widehat{S}})^{(1)}\big(\mathcal{K}(X\rtimes_\sigma\widehat{S})\big)
            \end{align*}
        by Remark~\ref{Sec.5.Rmk.to.Cor.to.Diag} and Proposition~\ref{embedding XxShat to OxShat is a Toep.rep.}.

        In much the same way as the calculation \eqref{Sec.4.in.the.Pf.thm} in the proof of Theorem~\ref{crossed product correspondences}, we see that
            \[\delta_\iota(J_X)(1_{M(A)}\otimes\widehat{S})\subseteq M(A\otimes\mathscr{K};J_X\otimes\mathscr{K})\]
        since $J_X$ is weakly $\delta$-invariant. It therefore follows by Proposition \ref{embedding XxShat to OxShat is a Toep.rep.}, Lemma~\ref{Lemma.for.Ext.Main.Thm.of.KQR2}, and the above equality that
        \begin{multline*}
            k_A\rtimes_\delta{\rm id}_{\widehat{S}}\big(\delta_\iota(J_X)(1_{M(A)}\otimes\widehat{S})\big) \\
            \begin{aligned}
            &= \overline{k_A\otimes{\rm id}_{\mathscr{K}}}\big(\delta_\iota(J_X)(1_{M(A)}\otimes\widehat{S})\big) \\
            &= \overline{(k_X\otimes{\rm id}_{\mathscr{K}})^{(1)}}\circ\varphi_{M(A\otimes\mathscr{K})}\big(\delta_\iota(J_X)(1_{M(A)}\otimes\widehat{S})\big) \\
            &= \overline{(k_X\otimes{\rm id}_{\mathscr{K}})^{(1)}}\big(\sigma^{(1)}_\iota(\varphi_A(J_X))(1_{M(\mathcal{K}(X))}\otimes\widehat{S})\big) \\
            &\subseteq\overline{(k_X\otimes{\rm id}_{\mathscr{K}})^{(1)}}\big(\mathcal{K}(X)\rtimes_{\sigma^{(1)}}\widehat{S}\big) \\
            &=(k_X\rtimes_\sigma{\rm id}_{\widehat{S}})^{(1)}\big(\mathcal{K}(X\rtimes_\sigma\widehat{S})\big),
            \end{aligned}
            \end{multline*}
        where the third and last step come from the $\delta$-$\sigma^{(1)}$ equivariancy of $\varphi_A$ and equality $\mathcal{K}(X)\rtimes_{\sigma^{(1)}}\widehat{S}=\mathcal{K}(X\rtimes_\sigma\widehat{S})$, respectively. This establishes the proposition.
    \end{proof}

    \begin{rmk}\label{Sec.5.Improve.KQRo2}\rm
        Recall from Definition 3.1 and Lemma 3.2 of \cite{KQRo1} that a nondegenerate correspondence homomorphism $(\psi,\pi):(X,A)\rightarrow(M(Y),M(B))$ is \emph{Cuntz-Pimsner covariant} if $\psi(X)\subseteq M_B(Y)$ and $\pi(J_X)\subseteq M(B;J_Y)$. Corollary~\ref{Sec.4.jxja} and Proposition \ref{Ext.Main.Thm.of.KQR2} then assure us that the representation $(j_X^\sigma,j_A^\delta)$ is always Cuntz-Pimsner covariant since \eqref{Sec.5.Drop.CP.KQR2} is obviously equivalent to
            \[j_A^\delta(J_X)\subseteq M(A\rtimes_\delta\widehat{S};J_{X\rtimes_\sigma\widehat{S}})\]
        which was a hypothesis of \cite[Theorem 4.4]{KQRo2} for $S=C^*(G)$. Therefore, Theorem~4.4 of \cite{KQRo2} can be improved as follows: if $(\sigma,\delta)$ is a nondegenerate coaction of a locally compact group $G$ on $(X,A)$ such that $\delta(J_X)\subseteq M(A\otimes C^*(G);J_X\otimes C^*(G))$, then we always have $\mathcal{O}_X\rtimes_\zeta G\cong \mathcal{O}_{X\rtimes_\sigma G}$.
    \end{rmk}

    \begin{thm}\label{Equiv.of.cov.}
        Let $(\sigma,\delta)$ be a coaction of a reduced Hopf $C^*$-algebra $S$ on a $C^*$-cor\-re\-spon\-dence $(X,A)$ such that $J_X$ is weakly $\delta$-invariant. Then the following conditions are equivalent:
        \begin{itemize}
        \item[\rm (i)] The representation $(k_X\rtimes_\sigma{\rm id}_{\widehat{S}},k_A\rtimes_\delta{\rm id}_{\widehat{S}}):(X\rtimes_\sigma\widehat{S},A\rtimes_\delta\widehat{S})\rightarrow\mathcal{O}_X\rtimes_\zeta\widehat{S}$ is covariant.
        \item[\rm (ii)] The ideal $J_{X\rtimes_\sigma\widehat{S}}$ is contained in $M(A\otimes\mathscr{K};J_X\otimes\mathscr{K})$.
        \item[\rm (iii)] The product $J_{X\rtimes_\sigma\widehat{S}}\,(\ker\varphi_A\otimes\mathscr{K})$ is zero.
        \end{itemize}
    \end{thm}

    \begin{proof}
        ${\rm (i)}\Leftrightarrow{\rm (ii)}$: Suppose (i). Since $(k_X\rtimes_\sigma{\rm id}_{\widehat{S}},k_A\rtimes_\delta{\rm id}_{\widehat{S}})$ is injective, we have
            \[J_{X\rtimes_\sigma\widehat{S}}=(k_A\rtimes_\sigma{\rm id}_{\widehat{S}})^{-1}\big((k_X\rtimes_\sigma{\rm id}_{\widehat{S}})^{(1)}(\mathcal{K}(X\rtimes_\sigma\widehat{S}))\big)\]
        by the comment below \cite[Proposition 5.14]{Kat2}. The same reason shows
            \[J_{M_{A\otimes\mathscr{K}}(X\otimes\mathscr{K})}=(\overline{k_A\otimes{\rm id}_{\mathscr{K}}})^{-1}\big(\overline{k_X\otimes{\rm id}_{\mathscr{K}}}^{\,(1)}\big(\mathcal{K}(M_{A\otimes\mathscr{K}}(X\otimes\mathscr{K}))\big)\big)\]
        since $\mathscr{K}$ is nuclear and then $(\overline{k_X\otimes{\rm id}_{\mathscr{K}}},\overline{k_A\otimes{\rm id}_{\mathscr{K}}})$ is covariant by Corollary~\ref{Sec.3.prop.for.thm.in.Sec.5}. It thus follows that $J_{X\rtimes_\sigma\widehat{S}}\subseteq J_{M_{A\otimes\mathscr{K}}(X\otimes\mathscr{K})}$ by Remark~\ref{Sec.5.Rmk.to.Cor.to.Diag} and Proposition \ref{embedding XxShat to OxShat is a Toep.rep.}. But, the latter is contained in $M(A\otimes\mathscr{K};J_X\otimes\mathcal{K})$ again by Corollary \ref{Sec.3.prop.for.thm.in.Sec.5}. This proves ${\rm (i)}\Rightarrow{\rm (ii)}$. Conversely, suppose (ii). Restricting the equality \eqref{Covariancy.of.the.Strict.Ext.1} of Lemma~\ref{Lemma.for.Ext.Main.Thm.of.KQR2} to the subalgebra $J_{X\rtimes_{\sigma}\widehat{S}}$, we can write
            \[(k_X\rtimes_\sigma{\rm id}_{\widehat{S}})^{(1)}\circ\varphi_{A\rtimes_\delta\widehat{S}}=k_A\rtimes_\delta{\rm id}_{\widehat{S}},\]
        which verifies ${\rm (ii)}\Rightarrow{\rm (i)}$.

        ${\rm(ii)}\Leftrightarrow{\rm(iii)}$: Assuming (ii) we have
            \begin{align*}
            J_{X\rtimes_{\sigma}\widehat{S}}(\ker\varphi_A\otimes\mathscr{K})
            &=J_{X\rtimes_{\sigma}\widehat{S}}(A\otimes\mathscr{K})(\ker\varphi_A\otimes\mathscr{K}) \\
            &\subseteq(J_X\otimes\mathscr{K})(\ker\varphi_A\otimes\mathscr{K})=0,
            \end{align*}
        and hence we get (iii). Finally, we always have
            \[\varphi_{A\otimes\mathscr{K}}\big((A\otimes\mathscr{K})J_{X\rtimes_\sigma\widehat{S}}\big)
            \subseteq\varphi_{A\otimes\mathscr{K}}(A\otimes\mathscr{K})\mathcal{K}(X\rtimes_\sigma\widehat{S})
            \subseteq \mathcal{K}(X\otimes\mathscr{K})\]
        by Remark \ref{Sec.5.Rmk.to.Cor.to.Diag}. Since $\ker\varphi_{A\otimes\mathscr{K}}=\ker(\varphi_A\otimes{\rm id}_{\mathscr{K}})=\ker\varphi_A\otimes\mathscr{K}$ by the exactness of $\mathscr{K}$, (iii) implies
            \[\big((A\otimes\mathscr{K})J_{X\rtimes_\sigma\widehat{S}}\big)\ker\varphi_{A\otimes\mathscr{K}}
            =(A\otimes\mathscr{K})\big(J_{X\rtimes_\sigma\widehat{S}}(\ker\varphi_A\otimes\mathscr{K})\big)=0.\]
        Therefore $(A\otimes\mathscr{K})J_{X\rtimes_\sigma\widehat{S}}\subseteq J_{X\otimes\mathscr{K}}$. But $J_{X\otimes\mathscr{K}}=J_X\otimes\mathscr{K}$ by Corollary~\ref{Sec.3.prop.for.thm.in.Sec.5}, which proves ${\rm(iii)}\Rightarrow{\rm(ii)}$.
    \end{proof}

    \begin{cor}\label{Sec.5.Cor.varphi.inj.}
        Let $(\sigma,\delta)$ be a coaction of $S$ on $(X,A)$ such that $J_X$ is weakly $\delta$-invariant. Assume that either (i) the ideal $J_{X\rtimes_\sigma\widehat{S}}$ of $A\rtimes_\delta\widehat{S}$ is generated by $\delta_\iota(J_X)$  
        or (ii) $\varphi_A$ is injective. Then $(k_X\rtimes_\sigma{\rm id}_{\widehat{S}},k_A\rtimes_\delta{\rm id}_{\widehat{S}})$ is covariant.
    \end{cor}

    \begin{proof}
        Assume (i), that is, $J_{X\rtimes_\sigma\widehat{S}}=\overline{(1_{M(A)}\otimes\widehat{S})\delta_\iota(J_X)(1_{M(A)}\otimes\widehat{S})}$. The nondegeneracy of $S$ and $\widehat{S}$ shows that
        \begin{align*}
            J_{X\rtimes_\sigma\widehat{S}}\,(A\otimes\mathscr{K}) 
            &=\overline{(1_{M(A)}\otimes\widehat{S})\delta_\iota(J_X)(1_{M(A)}\otimes\widehat{S})\,(A\otimes\mathscr{K})} \\
            &=\overline{(1_{M(A)}\otimes\widehat{S})\delta_\iota(J_X)\,(A\otimes\mathscr{K})} \\
            &=\overline{(1_{M(A)}\otimes\widehat{S})\delta_\iota(J_X)(1_{M(A)}\otimes S)\,(A\otimes\mathscr{K})} \\
            &\subseteq\overline{(1_{M(A)}\otimes\widehat{S})(J_X\otimes S)(A\otimes\mathscr{K})} =J_X\otimes\mathscr{K},
        \end{align*}
        in which the last inclusion follows from the weak $\delta$-invariancy of $J_X$. Hence we get the equivalent condition (ii) in Theorem~\ref{Equiv.of.cov.}. Evidently, assumption (ii) implies (iii) in Theorem~\ref{Equiv.of.cov.}.
    \end{proof}

    \begin{cor}\label{Sec.5.Cor.delta.trivial}
        Let $(\sigma,\delta)$ be a coaction of $S$ on $(X,A)$ such that $\delta$ is trivial, that is, $\delta(a)=a\otimes1_{M(S)}$ for $a\in A$. If the triple $(J_X,A,\widehat{S})$ satisfies the slice map property, then $(k_X\rtimes_\sigma{\rm id}_{\widehat{S}},k_A\rtimes_\delta{\rm id}_{\widehat{S}})$ is covariant. Moreover, $J_{X\rtimes_\sigma\widehat{S}}=J_X\otimes\widehat{S}$.
    \end{cor}

    \begin{proof}
        Since $\delta$ is trivial, $J_X$ is evidently weakly $\delta$-invariant. Hence the representation $(k_X\rtimes_\sigma{\rm id}_{\widehat{S}},k_A\rtimes_\delta{\rm id}_{\widehat{S}})$ on $\mathcal{O}_X\rtimes_\zeta\widehat{S}$ makes sense by Proposition~\ref{embedding XxShat to OxShat is a Toep.rep.}. To show that it is covariant, we check the equivalent condition (iii) in Theorem~\ref{Equiv.of.cov.}. First note that $\varphi_{A\rtimes_\delta\widehat{S}}=\varphi_{A\otimes\widehat{S}}=\varphi_A\otimes{\rm id}_{\widehat{S}}$. Then
        \begin{equation*}\label{Ker.for.Trivial}
            \ker\varphi_{A\rtimes_\delta\widehat{S}}=\ker(\varphi_A\otimes{\rm id}_{\widehat{S}})=\ker\varphi_A\otimes\widehat{S}
        \end{equation*}
        by Remarks \ref{S.M.P.}.(1). Since $\widehat{S}$ is a nondegenerate subalgebra of $\mathcal{L}(\mathcal{H})$, it follows that
            \[J_{X\rtimes_\sigma\widehat{S}}(\ker\varphi_A\otimes\mathscr{K})
            =\big(J_{X\rtimes_\sigma\widehat{S}}(\ker\varphi_A\otimes\widehat{S})\big)(1_{M(A)}\otimes\mathscr{K})=0,\]
        and therefore $(k_X\rtimes_\sigma{\rm id}_{\widehat{S}},k_A\rtimes_\delta{\rm id}_{\widehat{S}})$ is covariant.

        Let $\omega\in\mathcal{L}(\mathcal{H})_*$ and $T\in\mathscr{K}$. Applying the slice map ${\rm id}_A\otimes(\omega T)$ to $J_{X\rtimes_\sigma\widehat{S}}$ and then multiplying $a\in A$ yields
            \[a({\rm id}_A\otimes(\omega T))(J_{X\rtimes_\sigma\widehat{S}})=({\rm id}_A\otimes\omega)\big((a\otimes T)J_{X\rtimes_\sigma\widehat{S}}\big)\subseteq J_X,\]
        in which the last inclusion is due to the equivalent condition (ii) of Theorem~\ref{Equiv.of.cov.}. We thus have $({\rm id}_A\otimes\omega)(J_{X\rtimes_\sigma\widehat{S}})\subseteq J_X$ for $\omega\in\mathcal{L}(\mathcal{H})_*$, and conclude by Remarks~\ref{S.M.P.}.(2) that $J_{X\rtimes_\sigma\widehat{S}}\subseteq F(J_X,A,\widehat{S})=J_X\otimes\widehat{S}$. The converse follows from Proposition~\ref{Ext.Main.Thm.of.KQR2}.
    \end{proof}

    We now state and prove our main theorem.

    \begin{thm}\label{Main.Theorem.}
        Let $(\sigma,\delta)$ be a coaction of a reduced Hopf $C^*$-algebra $S$ on a $C^*$-cor\-re\-spon\-dence $(X,A)$ such that $J_X$ is weakly $\delta$-invariant. Suppose that the representation $(k_X\rtimes_\sigma{\rm id}_{\widehat{S}},k_A\rtimes_\delta{\rm id}_{\widehat{S}})$ is covariant. Then the integrated form
            \[(k_X\rtimes_\sigma{\rm id}_{\widehat{S}})\times(k_A\rtimes_\delta{\rm id}_{\widehat{S}}):\mathcal{O}_{X\rtimes_\sigma\widehat{S}}\rightarrow\mathcal{O}_X\rtimes_\zeta\widehat{S}\]
        is a surjective isomorphism.
    \end{thm}

    \begin{proof}
        Set $\Psi=(k_X\rtimes_\sigma{\rm id}_{\widehat{S}})\times(k_A\rtimes_\delta{\rm id}_{\widehat{S}})$. Note that the embedding $k_A\rtimes_\delta{\rm id}_{\widehat{S}}$ is clearly nondegenerate, and hence $\Psi$ is also nondegenerate.

        We claim that $\Psi(\mathcal{O}_{X\rtimes_\sigma\widehat{S}})$ contains all the elements of the form
            \[(1_{M(\mathcal{O}_X)}\otimes x)\big(\zeta_\iota\big(k_X(\xi_1)\cdots k_X(\xi_n)k_X(\eta_m)^*\cdots k_X(\eta_1)^*\big)\big)(1_{M(\mathcal{O}_X)}\otimes y)\]
        for nonnegative integers $m$ and $n$, vectors $\xi_1,\ldots,\xi_n,\eta_1,\ldots,\eta_m\in X$, and $x,y\in\widehat{S}$. This will prove that $\Psi$ is surjective  (\cite[Proposition 2.7]{Kat}). Since
            \[\zeta_\iota(k_A(A))(1_{M(\mathcal{O}_X)}\otimes\widehat{S})\subseteq\Psi(\mathcal{O}_{X\rtimes_\sigma\widehat{S}})\]
        by \eqref{Sec.5.Rep.1} of Proposition \ref{embedding XxShat to OxShat is a Toep.rep.}, we only show by considering adjoints that
            \begin{equation}\label{Sec.5.Pf.Main.Thm}
                (1_{M(\mathcal{O}_X)}\otimes x)\zeta_\iota(k_X(\xi_1)\cdots k_X(\xi_n))\in\Psi(\mathcal{O}_{X\rtimes_\sigma\widehat{S}})
            \end{equation}
        for positive integers $n$, vectors $\xi_1,\ldots,\xi_n\in X$, and $x\in\widehat{S}$. We now proceed by induction on $n$. For $n=1$, \eqref{Sec.5.Pf.Main.Thm} follows from the last equality of \eqref{Sec.5.Rep.1}. Suppose that \eqref{Sec.5.Pf.Main.Thm} is true for an $n$. Let $\xi,\xi_1,\ldots,\xi_n$ be $n+1$ vectors in $X$ and $x\in\widehat{S}$. Take an element $C\in\mathcal{O}_{X\rtimes_\sigma\widehat{S}}$ such that
            \[\Psi(C)=(1_{M(\mathcal{O}_X)}\otimes x)\zeta_\iota\big(k_X(\xi_1)k_X(\xi_2)\cdots k_X(\xi_n)\big).\]
        By Remarks \ref{Sec.4.Rmk.for.Pf.Main.Thm}.(2), we have
            \[\overline{k_{X\rtimes_\sigma\widehat{S}}}(j_X^\sigma(\xi))\in M_{A\rtimes_\delta\widehat{S}}(\mathcal{O}_{X\rtimes_\sigma\widehat{S}}).\]
        We claim that
            \begin{equation}\label{Sec.5.Pf.Main.Thm-2}
                \overline{\Psi}\big(\overline{k_{X\rtimes_\sigma\widehat{S}}}(j_X^\sigma(\xi))\big)=j_{\mathcal{O}_X}^\zeta(k_X(\xi)),
            \end{equation}
        where $j_{\mathcal{O}_X}^\zeta:\mathcal{O}_X\rightarrow M(\mathcal{O}_X\rtimes_\zeta\widehat{S})$ is the canonical homomorphism such that $j_{\mathcal{O}_X}^\zeta(c)D=\zeta_\iota(c)D$ for $c\in\mathcal{O}_X$ and $D\in\mathcal{O}_X\rtimes_\zeta\widehat{S}$. In fact, for
            \[v=\Psi\big(k_{A\rtimes_\delta\widehat{S}}\big(\delta_\iota(a)(1_{M(A)}\otimes x)\big)\big)=\zeta_\iota(k_A(a))(1_{M(\mathcal{O}_X)}\otimes x),\]
        we have
            \begin{align*}
                \overline{\Psi}\big(\overline{k_{X\rtimes_\sigma\widehat{S}}}(j_X^\sigma(\xi))\big)\,v
                &=\Psi\big(\overline{k_{X\rtimes_\sigma\widehat{S}}}(j_X^\sigma(\xi))\,k_{A\rtimes_\delta\widehat{S}}\big(\delta_\iota(a)(1_{M(A)}\otimes x)\big)\big) \\
                &=\Psi\big(k_{X\rtimes_\sigma\widehat{S}}\big(j_X^\sigma(\xi)\cdot(\delta_\iota(a)(1_{M(A)}\otimes x))\big)\big) \\
                &=\Psi\big(k_{X\rtimes_\sigma\widehat{S}}\big(\sigma_\iota(\xi\cdot a)\cdot(1_{M(A)}\otimes x)\big)\big) \\
                &=\zeta_\iota\big(k_X(\xi\cdot a)\big)(1_{M(\mathcal{O}_X)}\otimes x) \\
                &=j_{\mathcal{O}_X}^\zeta(k_X(\xi))\,\zeta_\iota(k_A(a))(1_{M(\mathcal{O}_X)}\otimes x)
                =j_{\mathcal{O}_X}^\zeta(k_X(\xi))\,v,
            \end{align*}
        which verifies the equality \eqref{Sec.5.Pf.Main.Thm-2} since $k_A\rtimes_\delta{\rm id}_{\widehat{S}}$ is nondegenerate. It is now obvious that for the product $C\,\overline{k_{X\rtimes_\sigma\widehat{S}}}(j_X^\sigma(\xi))\in\mathcal{O}_{X\rtimes_\sigma\widehat{S}}$ we have
            \[\Psi\big(C\,\overline{k_{X\rtimes_\sigma\widehat{S}}}(j_X^\sigma(\xi))\big)
            =(1_{M(\mathcal{O}_X)}\otimes x)\zeta_\iota\big(k_X(\xi_1)\cdots k_X(\xi_n)k_X(\xi)\big).\]
        Consequently, the statement \eqref{Sec.5.Pf.Main.Thm} is shown to be true for all positive integer $n$, and hence $\Psi$ is surjective.

        Let $\beta:\mathbb{T}\rightarrow Aut(\mathcal{O}_X)$ be the gauge action. Then the strict extensions $\overline{\beta_z\otimes{\rm id}_{\mathscr{K}}}$ are automorphisms on $M(\mathcal{O}_X\otimes\mathscr{K})$. We have
        \begin{equation*}\label{EqualityforGIUT1}
        \begin{aligned}
            \overline{\beta_z\otimes{\rm id}_{\mathscr{K}}}\big(\zeta_\iota(k_X(\xi))\big)&=z\zeta_\iota(k_X(\xi)), \\
            \overline{\beta_z\otimes{\rm id}_{\mathscr{K}}}\big(\zeta_\iota(k_A(a))\big)&=\zeta_\iota(k_A(a))
        \end{aligned}
        \end{equation*}
        in the same way as the last part of the proof of Theorem \ref{induced coactions on O_X}. Therefore,
            \begin{multline*}
            \overline{\beta_z\otimes{\rm id}_{\mathscr{K}}}\big((k_X\rtimes_\sigma{\rm id}_{\widehat{S}})\big(\sigma_\iota(\xi)\cdot(1_{M(A)}\otimes x)\big)\big) \\
            \begin{aligned}
            &=\overline{\beta_z\otimes{\rm id}_{\mathscr{K}}}\big(\zeta_\iota(k_X(\xi))(1_{M(\mathcal{O}_X)}\otimes x)\big) \\
            &=z\,\zeta_\iota(k_X(\xi))(1_{M(\mathcal{O}_X)}\otimes x) \\
            &=z\,(k_X\rtimes_\sigma{\rm id}_{\widehat{S}})\big(\sigma_\iota(\xi)\cdot(1_{M(A)}\otimes x)\big)
            \end{aligned}
            \end{multline*}
        and similarly
            \[\overline{\beta_z\otimes{\rm id}_{\mathscr{K}}}\big((k_A\rtimes_\delta{\rm id}_{\widehat{S}})\big(\delta_\iota(a)(1_{M(A)}\otimes x)\big)\big)=(k_A\rtimes_\delta{\rm id}_{\widehat{S}})\big(\delta_\iota(a)(1_{M(A)}\otimes x)\big).\]
        This proves that the restriction $\overline{\beta_z\otimes{\rm id}_{\mathscr{K}}}|_{\mathcal{O}_X\rtimes_\zeta\widehat{S}}$ defines an automorphism on $\mathcal{O}_X\rtimes_\zeta\widehat{S}=\Psi(\mathcal{O}_{X\rtimes_\sigma\widehat{S}})$, and the injective covariant representation $(k_X\rtimes_\sigma{\rm id}_{\widehat{S}},k_A\rtimes_\delta{\rm id}_{\widehat{S}})$ admits a gauge action. We thus conclude by \cite[Theorem~6.4]{Kat} that $\Psi$ is injective as well, which completes the proof.
    \end{proof}

    Applying Theorem \ref{Main.Theorem.} to group actions we can extend Theorem 2.10 of \cite{HaoNg} as Corollary~\ref{Cor1.to.Main.Thm} states below, the proof of which will be given in Appendix B. Let $(\gamma,\alpha)$ be an action of a locally compact group $G$ on $(X,A)$. By Theorem \ref{actions=coactions}, $(\gamma,\alpha)$ defines a coaction $(\sigma^\gamma,\delta^\alpha)$ of $C_0(G)$ on $(X,A)$, which induces a coaction $\zeta$ of $C_0(G)$ on $\mathcal{O}_X$ by Theorem \ref{induced coactions on O_X} and Remarks~\ref{Sec.3.Rmk.to.Thm}.(1). Let $\beta^\zeta$ be the action of $G$ on $\mathcal{O}_X$ corresponding to the coaction $\zeta$. In a similar way to \cite[Corollary~2.9]{HaoNg}, we define a representation
        \[(k_X\rtimes_\gamma G,k_A\rtimes_\alpha G):(X\rtimes_{\gamma,r}G,A\rtimes_{\alpha,r}G)\rightarrow\mathcal{O}_X\rtimes_{\beta^\zeta,r}G\]
    by
        \[(k_X\rtimes_\gamma G)(f)(r)=k_X(f(r)),\quad(k_A\rtimes_\alpha G)(g)(r)=k_A(g(r))\]
    for $f\in C_c(G,X)$, $g\in C_c(G,A)$, and $r\in G$.

    \begin{cor}\label{Cor1.to.Main.Thm}
        Let $(\gamma,\alpha)$ be an action of a locally compact group $G$ on $(X,A)$. If the representation $(k_X\rtimes_\gamma G,k_A\rtimes_\alpha G)$ is covariant, then its integrated form $(k_X\rtimes_\gamma G)\times(k_A\rtimes_\alpha G):\mathcal{O}_{X\rtimes_{\gamma,r}G}\rightarrow\mathcal{O}_X\rtimes_{\beta^\zeta,r}G$ is a surjective isomorphism.
    \end{cor}

    For the amenability in the next theorem, we refer to \cite{BS}. See also \cite{Ng}.

    \begin{thm}\label{amen.V}
        Let $(\sigma,\delta)$ be a coaction on $(X,A)$ of a reduced Hopf $C^*$-algebra $S$ defined by an amenable regular multiplicative unitary such that $J_X$ is weakly $\delta$-invariant. If $A$ is nuclear (or exact, respectively), then the same is true for $\mathcal{O}_X\rtimes_{\zeta}\widehat{S}$.
    \end{thm}

    \begin{proof}
        If $A$ is nuclear (or exact, respectively), then so is $A\rtimes_\delta\widehat{S}$ by \cite[Theorem~3,4]{Ng} (or by \cite[Theorem 3.13]{Ng}, respectively). Hence, the Toeplitz algebra $\mathcal{T}_{X\rtimes_{\sigma}\widehat{S}}$ is nuclear by \cite[Corollary 7.2]{Kat} (or exact by \cite[Theorem 7.1]{Kat}, respectively). Since nuclearity or exactness passes to quotients, it suffices to show that the representation $(k_X\rtimes_\sigma{\rm id}_{\widehat{S}},k_A\rtimes_\delta{\rm id}_{\widehat{S}})$ gives rise to a surjection from $\mathcal{T}_{X\rtimes_\sigma\widehat{S}}$ onto $\mathcal{O}_X\rtimes_\zeta\widehat{S}$. The proof of this then goes parallel to the one given in the proof of Theorem \ref{Main.Theorem.} using the embedding $(i_{X\rtimes_\sigma\widehat{S}},i_{A\rtimes_\delta\widehat{S}})$ of $(X\rtimes_\sigma\widehat{S},A\rtimes_\delta\widehat{S})$ into $\mathcal{T}_{X\rtimes_\sigma\widehat{S}}$ instead of $(k_{X\rtimes_\sigma\widehat{S}},k_{A\rtimes_\delta\widehat{S}})$ used in there.
    \end{proof}

    Applying the previous results we consider coactions on crossed products by $\mathbb{Z}$ which form an important example of Cuntz-Pimsner algebras.

    Let $\varphi$ be an automorphism on a $C^*$-algebra $A$. Equipped with the left action $\varphi_A(a)b=\varphi(a)b$ for $a,b\in A$, the Hilbert $A$-module $A$ then becomes a $C^*$-cor\-re\-spond\-ence (\cite[Examples~(3)]{Pim}), which we call a \emph{$\varphi$-cor\-re\-spond\-ence} and denote by $A(\varphi)$. For a $\varphi$-cor\-re\-spond\-ence $A(\varphi)$, it is clear that the multiplier correspondence $M(A(\varphi))$ coincides with the $\overline{\varphi}$-cor\-re\-spond\-ence $M(A)(\overline{\varphi})$ and the strict topology on $M(A(\varphi))$ is the usual one on the multiplier algebra $M(A)$.

    We want to consider a coaction of a Hopf $C^*$-algebra on a $\varphi$-cor\-re\-spond\-ence $(A(\varphi),A)$ and its induced coaction on $\mathcal{O}_{A(\varphi)}$. Before that, let us observe the following.

    \begin{lem}\label{Corr.Hom.bet'nIdentityCorr's}
        Let $\varphi$ and $\varphi'$ be automorphisms on $C^*$-algebras $A$ and $B$, respectively, and $\pi:A\rightarrow M(B)$ be a nondegenerate homomorphism. Let $v\in M(B)$ be a unitary such that
        \[v\,\pi(\varphi(a))=\overline{\varphi'}(\pi(a))\,v\quad(a\in A).\]
        Define
            \[\psi(a):=v\pi(a)\quad(a\in A).\]
        Then $(\psi,\pi):(A(\varphi),A)\rightarrow(M(B(\varphi')),M(B))$ is a nondegenerate correspondence homomorphism. Moreover, every nondegenerate correspondence homomorphism from $(A(\varphi),A)$ into $(M(B(\varphi')),M(B))$ is of this form.
    \end{lem}

    \begin{proof}
        For $a,a'\in A$,
         \[   \psi(\varphi(a)a')
            = v\pi(\varphi(a)a')=v\pi(\varphi(a))\,\pi(a')
            =\overline{\varphi'}(\pi(a))\,v\pi(a')=\overline{\varphi'}(\pi(a))\,\psi(a') \]
        and $\langle\psi(a),\psi(a')\rangle_{M(B)}=\pi(a)^*v^*v\pi(a')=\pi(\langle a,a'\rangle_A)$. Hence $(\psi,\pi)$ is a correspondence homomorphism, and obviously nondegenerate.

        For the converse, let $(\psi,\pi):(A(\varphi),A)\rightarrow(M(B(\varphi')),M(B))$ be a nondegenerate correspondence homomorphism, and consider its strict extension $(\overline{\psi},\overline{\pi})$. Let $v=\overline{\psi}(1_{M(A)})$. Since $(\overline{\psi},\overline{\pi})$ is a correspondence homomorphism, we have
            \[v^*v=\langle v,v\rangle_{M(B)}=\overline{\pi}(\langle1_{M(A)},1_{M(A)}\rangle_{M(A)})=1_{M(B)}.\]
        We also have
            \[    vv^*(\psi(a) b)
                =\overline{\psi}(1_{M(A)})\langle\overline{\psi}(1_{M(A)}),\psi(a) b\rangle_{M(B)}
                =\overline{\psi}(1_{M(A)})\pi(a)b=\overline{\psi}(a) b \]
        for $a\in A$ and $b\in B$ so that $vv^*=1_{M(B)}$. Hence $v$ is a unitary in $M(B)$. Finally,
            \[v\pi(\varphi(a))=\psi(\varphi(a))=\psi\big(\varphi(a)1_{M(A)}\big)=\overline{\varphi'}(\pi(a))v.\]
        This completes the proof.
    \end{proof}

    Let $\delta$ be a coaction of a Hopf $C^*$-algebra $(S,\Delta)$ on a $C^*$-algebra $A$ and $\varphi\in Aut(A)$. Let $v$ be a cocycle for the coaction $\delta$, that is, a unitary $v\in M(A\otimes S)$ satisfying
        \[v_{12}\,\overline{\delta\otimes{\rm id}_S}(v)=\overline{{\rm id}_A\otimes\Delta}(v)\]
    (\cite[Definition 0.4]{BS}), and suppose that
    \begin{equation}\label{EquivarianceforA,Aw.varphi}
            v\,\delta(\varphi(a))=\overline{\varphi\otimes{\rm id}_S}(\delta(a))\,v\quad(a\in A).
        \end{equation}
    Define $\sigma:A(\varphi)\rightarrow M(A\otimes S(\varphi\otimes{\rm id}_S))=M(A(\varphi)\otimes S)$ by
        \[\sigma(a):=v\delta(a)\quad(a\in A).\]
    Then $(\sigma,\delta)$ is a coaction of $S$ on the $\varphi$-cor\-re\-spond\-ence $(A(\varphi),A)$. Indeed, it is a nondegenerate correspondence homomorphism by Lemma \ref{Corr.Hom.bet'nIdentityCorr's}. Also, the computation
        \begin{align*}
            \overline{\sigma\otimes{\rm id}_S}(\sigma(a))
            &=v_{12}\,\overline{\delta\otimes{\rm id}_S}(v\delta(a)) \\
            &=v_{12}\,\overline{\delta\otimes{\rm id}_S}(v)\,\overline{\delta\otimes{\rm id}_S}(\delta(a)) \\
            &=\overline{{\rm id}_A\otimes\Delta}(v)\,\overline{{\rm id}_A\otimes\Delta}(\delta(a)) = \overline{{\rm id}_A\otimes\Delta}(\sigma(a))
        \end{align*}
        verifies the coaction identity of $\sigma$. The coaction nondegeneracy of $\delta$ gives
        \begin{align*}
            \overline{(1_{M(A)}\otimes S)\sigma(A)}
            &=\overline{(1_{M(A)}\otimes S)\,v\delta(A)}=\overline{(1_{M(A)}\otimes S)\,v\delta(\varphi(A))} \\
            &=\overline{(1_{M(A)}\otimes S)\,\big(\overline{\varphi\otimes{\rm id}_S}\,\delta(A)\big)\,v} \\
            &=\overline{\overline{\varphi\otimes{\rm id}_S}\big((1_{M(A)}\otimes S)\,\delta(A)\big)}\,v =(A\otimes S)v = A\otimes S
        \end{align*}
    so that $\sigma$ satisfies coaction nondegeneracy. Hence $(\sigma,\delta)$ is a coaction.

    The Cuntz-Pimsner algebra $\mathcal{O}_{A(\varphi)}$ is isomorphic to the crossed product $A\rtimes_\varphi\mathbb{Z}$ and an isomorphism $\mathcal{O}_{A(\varphi)}\cong A\rtimes_\varphi\mathbb{Z}$ can be given as follows. Let $(\pi,u)$ be the canonical covariant representation of the $C^*$-dynamical system $(A,\mathbb{Z},\varphi)$ on $M(A\rtimes_\varphi\mathbb{Z})$. Define $\psi:A(\varphi)\rightarrow A\rtimes_\varphi\mathbb{Z}$ by
        \[\psi(a)=u^{*}\pi(a)\quad(a\in A(\varphi)).\]
    It can be easily checked that $(\psi,\pi)$ is a covariant representation of $(A(\varphi),A)$ on $A\rtimes_\varphi\mathbb{Z}$. Furthermore, the integrated form $\psi\times\pi:\mathcal{O}_{A(\varphi)}\rightarrow A\rtimes_\varphi\mathbb{Z}$ gives a surjective isomorphism. We will identify in this way the universal covariant representations $(k_X,k_A)=(\psi,\pi)$ as well as the $C^*$-algebras $\mathcal{O}_{A(\varphi)}=A\rtimes_\varphi\mathbb{Z}$.

    Since $J_{A(\varphi)}=A$ is evidently weakly $\delta$-invariant, it follows by Theorem \ref{induced coactions on O_X} that $(\sigma,\delta)$ induces a coaction $\zeta$ of $S$ on $\mathcal{O}_{A(\varphi)}=A\rtimes_\varphi\mathbb{Z}$ which can be described explicitly on the canonical generators of $A\rtimes_\varphi\mathbb{Z}$ as follows. Theorem \ref{induced coactions on O_X} says that
        \[\zeta(\pi(a))=\overline{\pi\otimes{\rm id}_S}(\delta(a)),\]
        \[
                \zeta(u^{*}\pi(a))
                =\zeta(\psi(a))=\overline{\psi\otimes{\rm id}_S}(\sigma(a))
                =(u^{*}\otimes1_{M(S)})\,\overline{\pi\otimes{\rm id}_S}(v\delta(a))
        \]
    for $a\in A$. Note that $\overline{\zeta}(u^*)=(u^{*}\otimes1_{M(S)})\,\overline{\pi\otimes{\rm id}_S}(v)$. Hence,
        \begin{equation*}\label{Sec.3.Example1}
        \zeta(\pi(a)u^n)=\overline{\pi\otimes{\rm id}_S}(\delta(a))\big((u^{*}\otimes1_{M(S)})\,\overline{\pi\otimes{\rm id}_S}(v)\big)^{-n}
        \end{equation*}
    for $a\in A$ and $n\in\mathbb{Z}$.

    Assume now that the Hopf $C^*$-algebra $S$ is reduced. Then we can form the reduced crossed product correspondence $(A(\varphi)\rtimes_\sigma\widehat{S},A\rtimes_\delta\widehat{S})$ by Theorem~\ref{crossed product correspondences}. Let $v_\iota=\overline{{\rm id}_A\otimes\iota_S}(v)$. Since 
    the multiplication by $v_\iota$ from the left gives a Hilbert module isomorphism from $A\rtimes_\delta\widehat{S}$ onto $A(\varphi)\rtimes_\sigma\widehat{S}$, we may --- and do --- regard the $C^*$-cor\-re\-spond\-ence $A(\varphi)\rtimes_\sigma\widehat{S}$ as the Hilbert module $A\rtimes_\delta\widehat{S}$ with the left action
        \begin{equation}\label{Sec.4.Example1}
            \varphi_{A\rtimes_\delta\widehat{S}}(c)\,d=v_\iota^*\overline{\varphi\otimes{\rm id}_{\mathcal{K}(\mathcal{H})}}(c)v_\iota\,d
        \end{equation}
    for an element $c$ in the $C^*$-algebra $A\rtimes_\delta\widehat{S}$ and a vector $d$ in the Hilbert module $A\rtimes_\delta\widehat{S}$. Note that $\varphi_{A\rtimes_\delta\widehat{S}}$ is injective. Since $\varphi_A$ is injective, $\mathcal{O}_{A(\varphi)}\rtimes_\zeta\widehat{S}$ is the Cuntz-Pimsner algebra $\mathcal{O}_{A(\varphi)\rtimes_\sigma\widehat{S}}$ by Corollary \ref{Sec.5.Cor.varphi.inj.} and Theorem \ref{Main.Theorem.}.

    We can summerize what we have seen so far as follows.

    \begin{prop}\label{Sec.6.Prop.1}
        Let $\varphi$ be an automorphism on a $C^*$-algebra $A$ and $\delta$ be a coaction of a Hopf $C^*$-algebra $S$ on $A$. Let $v$ be a cocyle for $\delta$ satisfying \eqref{EquivarianceforA,Aw.varphi}. Define $\sigma:A(\varphi)\rightarrow M(A(\varphi)\otimes S)$ by $\sigma(a)=v\delta(a)$. Then the following hold.

        {\rm(i)} $(\sigma,\delta)$ is a coaction of $S$ on the $\varphi$-cor\-re\-spond\-ence $(A(\varphi),A)$.

        {\rm(ii)} Let $(\pi,u)$ be the canonical covariant representation of $(A,\mathbb{Z},\varphi)$ on $M(A\rtimes_\varphi\mathbb{Z})$. Then, the homomorphism \[\overline{\pi\otimes{\rm id}_S}\circ\delta:A\rightarrow M((A\rtimes_\varphi\mathbb{Z})\otimes S)\] and the unitary $\overline{\pi\otimes{\rm id}_S}(v^*)(u\otimes1_{M(S)})\in M((A\rtimes_\varphi\mathbb{Z})\otimes S)$ form a covariant representation of $(A,\mathbb{Z},\varphi)$ on $M((A\rtimes_\varphi\mathbb{Z})\otimes S)$ such that the integrated form gives a coaction $\zeta$ of $S$ on $A\rtimes_\varphi\mathbb{Z}$ and coincides with the coaction induced by $(\sigma,\delta)$.

        {\rm(iii)} If $S$ is reduced then $A(\varphi)\rtimes_\sigma\widehat{S}=A\rtimes_\delta\widehat{S}$ as Hilbert $(A\rtimes_\delta\widehat{S})$-modules and the left action is given by \eqref{Sec.4.Example1}. The reduced crossed product $(A\rtimes_\varphi\mathbb{Z})\rtimes_\zeta\widehat{S}=\mathcal{O}_{A(\varphi)}\rtimes_\zeta\widehat{S}$ is the Cuntz-Pimsner algebra $\mathcal{O}_{A(\varphi)\rtimes_\sigma\widehat{S}}$.
    \end{prop}

    We can say further if we take the cocycle $v$ in Proposition \ref{Sec.6.Prop.1} to be the identity. Let $v=1_{M(A\otimes S)}$. Then \eqref{EquivarianceforA,Aw.varphi} reduces to
        \[\delta\circ\varphi=\overline{\varphi\otimes{\rm id}_S}\circ\delta,\]
    and then $\varphi_{A\rtimes_\delta\widehat{S}}$ maps $A\rtimes_\delta\widehat{S}$ onto itself:
            \[
            \varphi_{A\rtimes_\delta\widehat{S}}\big(\delta_\iota(a)(1_{M(A)}\otimes x)\big)
            =\overline{\varphi\otimes{\rm id}_{\mathcal{K}(\mathcal{H})}}\big(\delta_\iota(a)(1_{M(A)}\otimes x)\big)
            =\delta_\iota(\varphi(a))(1_{M(A)}\otimes x)
            \]
    for $a\in A$ and $x\in\widehat{S}$. Hence $\varphi_{A\rtimes_\delta\widehat{S}}$ defines an automorphism $\varphi\rtimes{\rm id}$ on $A\rtimes_\delta\widehat{S}$ such that
        \[(\varphi\rtimes{\rm id})\big(\delta_\iota(a)(1_{M(A)}\otimes x)\big)=\delta_\iota(\varphi(a))(1_{M(A)}\otimes x)\]
    for $a\in A$ and $x\in\widehat{S}$. We thus see that $A(\varphi)\rtimes_\sigma\widehat{S}$ is the $(\varphi\rtimes{\rm id})$-cor\-re\-spond\-ence $A\rtimes_\delta\widehat{S}(\varphi\rtimes{\rm id})$. We have the equality
        \begin{equation}\label{Sec.5.Example1}
                \mathcal{O}_{A(\varphi)\rtimes_\sigma\widehat{S}}=\mathcal{O}_{A\rtimes_\delta\widehat{S}(\varphi\rtimes{\rm id})}=(A\rtimes_\delta\widehat{S})\rtimes_{\varphi\rtimes{\rm id}}\mathbb{Z}
        \end{equation}
    as well as
            \begin{equation}\label{Sec.5.Example2}
                \mathcal{O}_{A(\varphi)}\rtimes_\zeta\widehat{S}=(A\rtimes_\varphi\mathbb{Z})\rtimes_\zeta\widehat{S},
            \end{equation}
    and then have a surjective isomorphism
            \[\Psi:(A\rtimes_\delta\widehat{S})\rtimes_{\varphi\rtimes{\rm id}}\mathbb{Z}=\mathcal{O}_{A(\varphi)\rtimes_\sigma\widehat{S}}
            \rightarrow\mathcal{O}_{A(\varphi)}\rtimes_\zeta\widehat{S}=(A\rtimes_\varphi\mathbb{Z})\rtimes_\zeta\widehat{S}.\]
    Let us describe $\Psi$ on the canonical generators of the iterated crossed products $(A\rtimes_\delta\widehat{S})\rtimes_{\varphi\rtimes{\rm id}}\mathbb{Z}$ and $(A\rtimes_\varphi\mathbb{Z})\rtimes_\zeta\widehat{S}$. As $(\pi,u)$ in Proposition~\ref{Sec.6.Prop.1}, let $(\widetilde{\pi},\widetilde{u})$ be the canonical covariant representation of the $C^*$-dynamical system $(A\rtimes_\delta\widehat{S},\,\mathbb{Z},\,\varphi\rtimes{\rm id})$ on $M((A\rtimes_\delta\widehat{S})\rtimes_{\varphi\rtimes{\rm id}}\mathbb{Z})$. Let
        \begin{align*}
        d_1 &= k_{A(\varphi)\rtimes_\delta\widehat{S}}\big(\delta_\iota(a)\cdot(1_{M(A)}\otimes x)\big)\in \mathcal{O}_{A(\varphi)\rtimes_\sigma\widehat{S}}, \\
        d_2 &= \widetilde{u}^{*}\widetilde{\pi}\big(\delta_\iota(a)(1_{M(A)}\otimes x)\big)\in (A\rtimes_\delta\widehat{S})\rtimes_{\varphi\rtimes{\rm id}}\mathbb{Z}, \\
        d_3 &= \zeta_\iota(k_{A(\varphi)}(a))(1_{M(\mathcal{O}_{A(\varphi)})}\otimes x) \in \mathcal{O}_{A(\varphi)}\rtimes_\zeta\widehat{S}, \\
        d_4 &= \zeta_\iota(u^{*}\pi(a))(1_{M(A\rtimes_\varphi\mathbb{Z})}\otimes x)\in (A\rtimes_\varphi\mathbb{Z})\rtimes_\zeta\widehat{S}.
        \end{align*}
    We then have $d_1=d_2$ in \eqref{Sec.5.Example1}, and $d_3=d_4$ in \eqref{Sec.5.Example2}. Since $\Psi(d_1)=d_3$, we may write $\Psi(d_2)=d_4$. Note that $\overline{\Psi}(\widetilde{u})=\overline{\zeta_\iota}(u)$. Therefore
            \[\Psi\big(\widetilde{u}^n\,\widetilde{\pi}\big(\delta_\iota(a)(1_{M(A)}\otimes x)\big)\big)=\zeta_\iota(u^n\pi(a))(1_{M(A\rtimes_\varphi\mathbb{Z})}\otimes x),\]
    or equivalently, by the fact that $(\pi,u)$ and $(\widetilde{\pi},\widetilde{u})$ are covariant representations,
            \begin{equation}\label{Sec.6.Cor.1}
            \Psi\big(\widetilde{\pi}\big(\delta_\iota(a)(1_{M(A)}\otimes x)\big)\,\widetilde{u}^n\big)
            =\zeta_\iota(\pi(a)u^n)(1_{M(A\rtimes_\varphi\mathbb{Z})}\otimes x)
            \end{equation}
    for $a\in A$, $x\in\widehat{S}$, and $n\in\mathbb{Z}$. We summarize this in the next corollary.

    \begin{cor}
        Under the hypothesis and notation in Proposition \ref{Sec.6.Prop.1} with $v$ replaced by $1_{M(A\otimes S)}$, the formula
            \[\zeta(\pi(a)u^n)=\overline{\pi\otimes{\rm id}_S}(\delta(a))(u^n\otimes1_{M(S)})\quad(a\in A,\ n\in\mathbb{Z})\]
        defines a coaction $\zeta$ of $S$ on $A\rtimes_\varphi\mathbb{Z}$. Moreover, if $S$ is reduced then there exists a surjective isomorphism
            \[\Psi:(A\rtimes_\delta\widehat{S})\rtimes_{\varphi\rtimes{\rm id}}\mathbb{Z}
            \rightarrow(A\rtimes_\varphi\mathbb{Z})\rtimes_\zeta\widehat{S}\]
        between the iterated crossed products such that \eqref{Sec.6.Cor.1} holds.
    \end{cor}

\appendix

\section{Coactions of $C_0(G)$ on $(X,A)$}\label{App.A}

{\renewcommand*{\thesection}{\Alph{section}}

    The goal of this section is to show that there exists a one-to-one correspondence between actions of a locally compact group $G$ on $(X,A)$ in the sense of \cite{EKQR} and coactions of the commutative Hopf $C^*$-algebra $C_0(G)$ on $(X,A)$ (Theorem~\ref{actions=coactions}). 

    Let us fix some notations. Let $(X,A)$ be a nondegenerate $C^*$-cor\-re\-spond\-ence as before, and $G$ be a locally compact Hausdorff space. By $M(X)_s$ we mean the multiplier correspondence $M(X)$ endowed with the strict topology. We denote by $C_b(G,M(X)_s)$ the Banach space of all bounded continuous functions from $G$ to $M(X)_s$ with the sup-norm, and by $C_0(G,X)$ the closed subspace of $C_b(G,M(X)_s)$ consisting of functions with values in $X$ which are also norm continuous and vanishes at infinity.

    For an identity correspondence $(X,A)=(A,A)$, the Banach space $C_b(G,M(X)_s)$ becomes a $C^*$-algebra under the usual point-wise operations. In this case,
        \[C_b(G,M(X)_s)=M(X\otimes C_0(G))\]
    (\cite[Corollary 3.4]{APT}). We first generalize this in Theorem \ref{Cbstr=M} to nondegenerate $C^*$-cor\-re\-spond\-ences, which will enable one to prove the bijective correspondence between $G$-actions and $C_0(G)$-coactions on $(X,A)$.

    \begin{prop}\label{CbG,MXs}
        The Banach space $C_b(G,M(X)_s)$ is a $C^*$-cor\-re\-spond\-ence over $C_b(G,M(A)_s)$ with the following point-wise operations
            \begin{equation}\label{Appendix.A.BimoduleOperation}
            \begin{aligned}
                (m\cdot l)(r) &= m(r)\cdot l(r), \\
                \langle m,n\rangle_{C_b(G,M(A)_s)}(r) &= \langle m(r),n(r)\rangle_{M(A)}, \\
                \big(\varphi_{C_b(G,M(A)_s)}(l)\,m\big)(r) &= \varphi_{M(A)}(l(r))\,m(r)
            \end{aligned}
            \end{equation}
        for $m,n\in C_b(G,M(X)_s)$, $l\in C_b(G,M(A)_s)$, and $r\in G$.
    \end{prop}

    \begin{proof}
        Write $\varphi=\varphi_{C_b(G,M(A)_s)}$. The only part requiring proof is that the functions on \eqref{Appendix.A.BimoduleOperation} are strictly continuous. We prove this only for the function $\varphi(l)\,m$. The others can be handled in the same way. Let $\{r_i\}$ be a net in $G$ converging to an $r\in G$, $a\in A$, and $T\in\mathcal{K}(X)$. Evidently, $(\varphi(l)m)(r_i)\cdot a-(\varphi(l)m)(r)\cdot a$ converges to 0. Factor $T=T'\varphi_A(a')$ for some $T'\in\mathcal{K}(X)$ and $a'\in A$, which is possible by the Hewitt-Cohen factorization theorem (see for example \cite[Proposition 2.33]{RaWi}) since the left action $\varphi_A$ is nondegenerate. Then the difference
        \begin{align*}
            T(\varphi(l)m)(r_i)-T(\varphi(l)m)(r)
            &=\big(T'\varphi_A(a'l(r_i))\,m(r_i)-T'\varphi_A(a'l(r))\,m(r_i)\big) \\
            &\quad + \big(T\varphi_{M(A)}(l(r))\,m(r_i)-T\varphi_{M(A)}(l(r))\,m(r)\big)
        \end{align*}
        converges to 0 by the strict continuity of both $l$ and $m$ and also by the boundedness of $m$. Hence $\varphi(l)m$ is strictly continuous.
    \end{proof}

    It is clear that $(C_0(G,X),C_0(G,A))$ is also a $C^*$-cor\-re\-spon\-dence with respect to the restriction of operations \eqref{Appendix.A.BimoduleOperation}.

    We call a correspondence homomorphism $(\psi,\pi):(X,A)\rightarrow(Y,B)$ an \emph{isomorphism} if both $\psi$ and $\pi$ are bijective. In this case, $(X,A)$ and $(Y,B)$ are said to be \emph{isomorphic}. The next corollary is an easy consequence of Corollary \ref{O_(XotimesB)=O_XotimesB}.

    \begin{cor}\label{C0GX=XotimesC0X}
        The $C^*$-correspondence $(C_0(G,X),C_0(G,A))$ and the tensor product correspondence $(X\otimes C_0(G),A\otimes C_0(G))$ are isomorphic.
    \end{cor}

     \begin{lem}\label{Preparation.for.Applying.EKQR}
        With respect to the operations \eqref{Appendix.A.BimoduleOperation}, the following hold.
        \begin{itemize}
            \item[\rm(i)] $\varphi_{C_b(G,M(A)_s)}\big(C_b(G,M(A)_s)\big)\,C_0(G,X)= C_0(G,X)$,
            \item[\rm(ii)] $C_0(G,X)\cdot C_b(G,M(A)_s)= C_0(G,X)$,
            \item[\rm(iii)] $C_b(G,M(X)_s)\cdot C_0(G,A)= C_0(G,X)$.
        \end{itemize}
    \end{lem}

    \begin{proof}
        It is obvious that the inclusion $\supseteq$ holds on each of (i) and (ii). The same is true for (iii) by the Hewitt-Cohen factorization theorem since $(C_0(G,X),C_0(G,A))$ is isomorphic to the nondegenerate $C^*$-cor\-re\-spond\-ence $(X\otimes C_0(G),A\otimes C_0(G))$. For the inclusion $\subseteq$ in (i), let $l\in C_b(G,M(A)_s)$ and $x\in C_0(G,X)$, and write $x=\varphi_{C_0(G,A)}(f)y$ for some $f\in C_0(G,A)$ and $y\in C_0(G,X)$. Then
            \[\varphi_{C_b(G,M(A)_s)}(l)\,x=\varphi_{C_0(G,A)}(lf)\,y\in C_0(G,X),\]
        which proves (i). Similarly we have the inclusion $\subseteq$ in (ii). Finally, the triangle inequality verifies that the functions in the left-hand side space of (iii) are continuous, which gives $\subseteq$ in (iii).
    \end{proof}

    Henceforth, we identify $C_0(G,X)=X\otimes C_0(G)$ as well as $C_b(G,M(A)_s)=M(A\otimes C_0(G))$. The next theorem generalizes \cite[Corollary 3.4]{APT}.

    \begin{thm}\label{Cbstr=M}
        The map
            \[(\psi,{\rm id}):(C_b(G,M(X)_s),C_b(G,M(A)_s))\rightarrow(M(X\otimes C_0(G)),M(A\otimes C_0(G)))\]
        given by
            \begin{equation*}\label{Appendix.A.Cbstr=M}
                \psi(m)\cdot f=m\cdot f
            \end{equation*}
        for $m\in C_b(G,M(X)_s)$ and $f\in A\otimes C_0(G)$ is an isomorphism.
    \end{thm}

    \begin{proof}
        By Lemma \ref{Preparation.for.Applying.EKQR}, we can apply \cite[Proposition 1.28]{EKQR} to see that $(\psi,{\rm id})$ is an injective correspondence homomorphism. It thus remains to show that $\psi$ is surjective. Let $n\in M(X\otimes C_0(G))$. For each $r\in G$, define $m_n(r):A\rightarrow X$ and $m_n^*(r):X\rightarrow A$ by
            \begin{equation*}\label{Appendix.A.1}
                m_n(r)(a):=\big(n\cdot(a\otimes\phi_r)\big)(r),\quad m_n^*(r)(\xi):=\big(n^*(\xi\otimes\phi_r)\big)(r),
            \end{equation*}
        where $\phi_r\in C_c(G)$ such that $\phi_r\equiv1$ on a neighborhood of $r$. It is immaterial which $\phi_r$ we take to define $m_n(r)$ and $m_n^*(r)$ as long as $\phi_r\equiv1$ near $r$. Since
            \[\langle n\cdot(a\otimes\phi_r),\xi\otimes\phi_r\rangle_{A\otimes C_0(G)}=\langle a\otimes\phi_r,n^*(\xi\otimes\phi_r)\rangle_{A\otimes C_0(G)},\]
        we have $\langle m_n(r)\cdot a,\xi\rangle_A=\langle a,m_n^*(r)\xi\rangle_A$ by evaluating at $r$, and thus obtain a function $m_n:G\rightarrow M(X)$ with $m_n(r)^*=m_n^*(r)$. By definition, $\|m_n(r)\|\leq\|n\|$ for $r\in G$, and hence $m_n$ is bounded. To see that $m_n$ is strictly continuous, let $\{r_i\}$ be a net in $G$ converging to an $r\in G$, $a\in A$, and $\xi,\eta\in X$. Evidently, $\{m_n(r_i)\cdot a\}$ converges to $m_n(r)\cdot a$. The same is true for the net $\{m_n(r_i)^*\xi\}$, and hence $\{\theta_{\eta,\xi}m_n(r_i)\}=\{\theta_{\eta,m_n(r_i)^*\xi}\}$ converges to $\theta_{\eta,m_n(r)^*\xi}=\theta_{\eta,\xi}m_n(r)$, and consequently $\{Tm_n(r_i)\}$ converges to $Tm_n(r)$ for $T\in\mathcal{K}(X)$. Therefore $m_n\in C_b(G,M(X)_s)$. Finally, we have
            \[\big(n\cdot(a\otimes\phi_r)\big)(r)=m_n(r)\cdot a = \big(m_n\cdot(a\otimes\phi_r)\big)(r)\]
        for $a\in A$ and $r\in G$, which shows $\psi(m_n)=n$.
    \end{proof}

    Let $Aut(X,A)$ be the group of isomorphisms from $(X,A)$ onto itself. Recall from \cite[Definition 2.5]{EKQR} that an \emph{action} of a locally compact group $G$ on $(X,A)$ is a homomorphism $(\gamma,\alpha):G\rightarrow Aut(X,A)$ such that for each $\xi\in X$ and $a\in A$, the maps
        \[G\ni r\mapsto\gamma_r(\xi)\in X,\quad G\ni r\mapsto\alpha_r(a)\in A\]
    are both continuous. The formulas
    \begin{equation}\label{actions=coactions.formulars}
                \sigma^\gamma(\xi)(r)=\gamma_r(\xi),\quad\delta^\alpha(a)(r)=\alpha_r(a)
            \end{equation}
    then define elements $\sigma^\gamma(\xi)\in M(X\otimes C_0(G))$ and $\delta^\alpha(a)\in M(A\otimes C_0(G))$ for $\xi\in X$ and $a\in A$ by Theorem~\ref{Cbstr=M}. Note that $(\sigma^\gamma,\delta^\alpha)$ is by definition a correspondence homomorphism. The proof of the next theorem is omitted since the $C^*$-algebra counterpart is well-known (for example, \cite[Theorem~9.2.4]{Timm}) and readily transferred to the $C^*$-cor\-re\-spond\-ence setting with an aid of Theorem~\ref{Cbstr=M}.

    \begin{thm}\label{actions=coactions}
        The formulas \eqref{actions=coactions.formulars} define a one-to-one correspondence between actions of $G$ on $(X,A)$ and coactions of $C_0(G)$ on $(X,A)$.
    \end{thm}

}

\section{$C^*$-correspondences $X\rtimes_\sigma\widehat{S}_{\widehat{W}_G}$}\label{App.B}

{\renewcommand*{\thesection}{\Alph{section}}

    It is well-known that $\mathcal{L}_A(A\otimes\mathcal{H})=M(A\otimes\mathcal{K}(\mathcal{H}))$ for a $C^*$-algebra $A$ and a Hilbert space $\mathcal{H}$. We generalize this in Proposition \ref{Appendix.B.Prop.1} to a nondegenerate $C^*$-cor\-re\-spond\-ence:
        \[\mathcal{L}_A(A\otimes\mathcal{H},X\otimes\mathcal{H})=M(X\otimes\mathcal{K}(\mathcal{H})).\]
    Using this, we show in Corollary \ref{Appendix.B.Cor.1} that the construction of Theorem~\ref{crossed product correspondences} recovers the crossed product correspondence $(X\rtimes_{\gamma,r}G,A\rtimes_{\alpha,r}G)$ in \cite{EKQR} for an action $(\gamma,\alpha)$ of $G$ on $(X,A)$. 

    We first clarify the $C^*$-cor\-re\-spon\-dence $(\mathcal{L}_A(A\otimes\mathcal{H},X\otimes\mathcal{H}),\mathcal{L}_A(A\otimes\mathcal{H}))$ in the next lemma whose proof is trivial, and so we omit it.

    \begin{lem}\label{Appendix.B.Lemma0}
        The Banach space $\mathcal{L}_A(A\otimes\mathcal{H},X\otimes\mathcal{H})$ is a $C^*$-cor\-re\-spond\-ence over $\mathcal{L}_A(A\otimes\mathcal{H})$ with the following operations
        \begin{equation}\label{Appendix.B.Bimod.Operations}
            m\cdot l=m\circ l,\quad
            \langle m,n\rangle_{\mathcal{L}_A(A\otimes\mathcal{H})}=m^*\circ n,\quad
            \varphi_{\mathcal{L}_A(A\otimes\mathcal{H})}=\varphi_{M(A\otimes\mathcal{K}(\mathcal{H}))}
        \end{equation}
        for $m,n\in\mathcal{L}_A(A\otimes\mathcal{H},X\otimes\mathcal{H})$ and $l\in\mathcal{L}_A(A\otimes\mathcal{H})$.
    \end{lem}

    Note that $(\mathcal{K}_A(A\otimes\mathcal{H},X\otimes\mathcal{H}),\mathcal{K}_A(A\otimes\mathcal{H}))$ is also a $C^*$-cor\-re\-spond\-ence with  the restriction of the operations given in \eqref{Appendix.B.Bimod.Operations}.

    \begin{lem}
        There exists an isomorphism
            \[(\psi_0,{\rm id}):(\mathcal{K}_A(A\otimes\mathcal{H},X\otimes\mathcal{H}),\mathcal{K}_A(A\otimes\mathcal{H}))
            \rightarrow(X\otimes\mathcal{K}(\mathcal{H}),A\otimes\mathcal{K}(\mathcal{H}))\]
        such that $\psi_0(\theta_{\xi\otimes h,\,a\otimes k})=\xi\cdot a^*\otimes\theta_{h,k}$ for $\xi\in X$, $a\in A$, and $h,k\in\mathcal{H}$.
    \end{lem}

    \begin{proof}
        Let $\xi_i\in X$, $a_i\in A$, and $h_i,k_i\in\mathcal{H}$ for $i=1,\ldots,n$. We claim that that the norm of the operator $\sum_{i=1}^n\theta_{\xi_i\otimes h_i,\,a_i\otimes k_i}$ agrees with that of $\sum_{i=1}^n\xi_i\cdot a_i^*\otimes\theta_{h_i,k_i}$, which proves that $\psi_0$ is well-defined and isometric. For this, we may assume that the vectors $h_i$ are mutually orthonormal and similarly for $k_i$. Then
        \begin{align*}
            \big\|\sum_{i=1}^n\theta_{\xi_i\otimes h_i,\,a_i\otimes k_i}\big\|^2
            &= \big\|\sum_{i,j=1}^n\theta_{a_i\otimes k_i,\,\xi_i\otimes h_i}\theta_{\xi_j\otimes h_j,\,a_j\otimes k_j}\big\| \\
            &= \big\|\sum_{i,j=1}^n\theta_{(a_i\otimes k_i)\cdot\langle\xi_i\otimes h_i,\xi_j\otimes h_j\rangle_A,\,a_j\otimes k_j}\big\| \\
            &= \big\|\sum_{i=1}^n\theta_{a_i\langle\xi_i,\xi_i\rangle_A\otimes k_i,\,a_i\otimes k_i}\big\|.
        \end{align*}
        By \cite[Lemma 2.1]{KPWa}, the last of the above equalities coincides with the norm of the following product of two positive $n\times n$ matrices
            \[\Big(\big\langle a_i\langle\xi_i,\xi_i\rangle_A\otimes k_i,\,a_j\langle\xi_j,\xi_j\rangle_A\otimes k_j\big\rangle_A\Big)^{1/2}
                \Big(\langle a_i\otimes k_i,a_j\otimes k_j\rangle_A\Big)^{1/2}\]
        which is diagonal by orthogonality. Let
            \[b_i=\langle\xi_i\cdot a_i^*,\,\xi_i\rangle_A\quad(i=1,\ldots,n).\]
        Then
            \begin{align*}\label{Appendix.B.Eqn1.in.Lem}
                \big\|\sum_{i=1}^n\theta_{\xi_i\otimes h_i,\,a_i\otimes k_i}\big\|^2
                &= \max_{i=1,\ldots,n}\big\|\big(\langle\xi_i\cdot a_i^*,\,\xi_i\rangle_A^*\langle\xi_i\cdot a_i^*,\,\xi_i\rangle_A\big)^{1/2}(a_i^*a_i)^{1/2}\big\| \\
                &=\max_{i=1,\ldots,n}\|(b_i^*b_i)^{1/2}(a_i^*a_i)^{1/2}\|.
            \end{align*}
        On the other hand,
        \begin{align*}
            \big\|\sum_{i=1}^n\xi_i\cdot a_i^*\otimes\theta_{h_i,k_i}\big\|^2
            &= \big\|\sum_{i,j=1}^n\big\langle\xi_i\cdot a_i^*\otimes\theta_{h_i,k_i},\,\xi_j\cdot a_j^*\otimes\theta_{h_j,k_j}\big\rangle_{A\otimes\mathcal{K}(\mathcal{H})}\big\| \\
            &= \big\|\sum_{i,j=1}^n\langle\xi_i\cdot a_i^*,\,\xi_j\cdot a_j^*\rangle_A\otimes\theta_{k_i\langle h_i,h_j\rangle,\,k_j}\big\| \\
            &= \max_{i=1,\ldots,n}\|\langle\xi_i\cdot a_i^*,\,\xi_i\rangle_A\,a_i^*\|=\max_{i=1,\ldots,n}\|b_ia_i^*\|
        \end{align*}
        again by orthonormality. Our claim then follows since
        \begin{align*}
            \|(b_i^*b_i)^{1/2}(a_i^*a_i)^{1/2}\|^2
            &= \|(a_i^*a_i)^{1/2}b_i^*b_i(a_i^*a_i)^{1/2}\| \\
            &= \|b_i(a_i^*a_i)^{1/2}\|^2  = \|b_ia_i^*a_ib_i^*\|  = \|b_ia_i^*\|^2.
        \end{align*}
        The remaining parts of the lemma are now easily verified.
    \end{proof}

    In the next proposition, we identify $\mathcal{K}_A(A\otimes\mathcal{H},X\otimes\mathcal{H})=X\otimes\mathcal{K}(\mathcal{H})$. Note that for $m\in\mathcal{L}_A(A\otimes\mathcal{H},X\otimes\mathcal{H})$ and $f\in A\otimes\mathcal{K}(\mathcal{H})(=\mathcal{K}_A(A\otimes\mathcal{H}))$, the product $m\cdot f$ defines an element of $X\otimes\mathcal{K}(\mathcal{H})$.

    \begin{prop}\label{Appendix.B.Prop.1}
        There exists an isomorphism
            \[(\psi,{\rm id}):(\mathcal{L}_A(A\otimes\mathcal{H},X\otimes\mathcal{H}),\mathcal{L}_A(A\otimes\mathcal{H}))
            \rightarrow(M(X\otimes\mathcal{K}(\mathcal{H})),M(A\otimes\mathcal{K}(\mathcal{H})))\]
        such that
        \begin{equation*}\label{Appendix.B.Define.psi}
            \psi(m)\cdot f=m\cdot f
        \end{equation*}
        for $m\in\mathcal{L}_A(A\otimes\mathcal{H},X\otimes\mathcal{H})$ and $f\in A\otimes\mathcal{K}(\mathcal{H})$.
    \end{prop}

    \begin{proof}
        For the operations on \eqref{Appendix.B.Bimod.Operations}, the following can be easily seen to hold:
        \begin{itemize}
            \item[(i)] $\mathcal{K}_A(A\otimes\mathcal{H},X\otimes\mathcal{H})\cdot\mathcal{L}_A(A\otimes\mathcal{H})=
            \mathcal{K}_A(A\otimes\mathcal{H},X\otimes\mathcal{H})$;
            \item[(ii)] $\varphi_{\mathcal{L}_A(A\otimes\mathcal{H})}\big(\mathcal{L}_A(A\otimes\mathcal{H})\big)\,
                \mathcal{K}_A(A\otimes\mathcal{H},X\otimes\mathcal{H}) =\mathcal{K}_A(A\otimes\mathcal{H},X\otimes\mathcal{H})$;
            \item[(iii)] $\mathcal{L}_A(A\otimes\mathcal{H},X\otimes\mathcal{H})\cdot\mathcal{K}_A(A\otimes\mathcal{H})=
            \mathcal{K}_A(A\otimes\mathcal{H},X\otimes\mathcal{H})$.
        \end{itemize}
        Thus $(\psi,{\rm id})$ is an injective correspondence homomorphism by \cite[Proposition 1.28]{EKQR}. To see that $\psi$ is surjective, let $n\in M(X\otimes\mathcal{K}(\mathcal{H}))$. Take a net $\{x_i\}$ in $X\otimes\mathcal{K}(\mathcal{H})$ strictly converging to $n$. Then the limits $\lim_ix_ih$ and $\lim_ix_i^*k$ clearly exist for $h\in A\otimes\mathcal{H}$ and $k\in X\otimes\mathcal{H}$. Define $m_n:A\otimes\mathcal{H}\rightarrow X\otimes\mathcal{H}$ and $m_n^*:X\otimes\mathcal{H}\rightarrow A\otimes\mathcal{H}$ by
            \[m_nh=\lim_ix_ih,\quad m_n^*k=\lim_ix_i^*k.\]
        We see from
            \[\langle m_nh,k\rangle_A=\lim_i\langle x_ih,k\rangle_A=\lim_i\langle h,x_i^*k\rangle_A=\langle h,m_n^*k\rangle_A\]
        that $m_n\in\mathcal{L}_A(A\otimes\mathcal{H},X\otimes\mathcal{H})$ with the adjoint $m_n^*$. Obviously $\psi(m_n)=n$.
    \end{proof}

    From now on, we identify $\mathcal{L}_A(A\otimes\mathcal{H},X\otimes\mathcal{H})=M(X\otimes\mathcal{K}(\mathcal{H}))$.

    \begin{rmk}\label{Appendix.B.Rmk.EmbedA}\rm
        Let $\mu_G:C_0(G)\hookrightarrow\mathcal{L}(L^2(G))$ be the embedding in \eqref{Prel.pi.and.U}. The strict extension $\overline{{\rm id}_X\otimes\mu_G}$ then embeds $M(X\otimes C_0(G))$ into $M(X\otimes\mathcal{K}(L^2(G)))$ such that if $m\in C_b(G,M(X)_s)$ and $h\in C_c(G,A)\subseteq A\otimes L^2(G)$, then
            \begin{equation*}\label{Appendix.B.Rmk.1}
                (\overline{{\rm id}\otimes\mu_G}(m)h)(r)=m(r)\cdot h(r)\quad (r\in G)
            \end{equation*}
        by strict continuity.
    \end{rmk}

    Let $(\gamma,\alpha)$ be an action of a locally compact group $G$ on $(X,A)$. The \emph{crossed product correspondence} $(X\rtimes_{\gamma,r}G,A\rtimes_{\alpha,r}G)$ is the completion of the $C_c(G,A)$-bimodule $C_c(G,X)$ such that
        \begin{align*}
            (x\cdot f)(r) &=\int_Gx(s)\cdot\alpha_s(f(s^{-1}r))\,ds, \\
            \langle x,y\rangle_{A\rtimes_{\alpha,r}G}(r) &=\int_G\alpha_s^{-1}(\langle x(s),y(sr)\rangle_A)\,ds, \\
            \big(\varphi_{A\rtimes_{\alpha,r}G}(f)\,x\big)(r) &=\int_G\varphi_A(f(s))\,\gamma_s(x(s^{-1}r))\,ds
        \end{align*}
    for $x,y\in C_c(G,X)$, $f\in C_c(G,A)$, and $r\in G$ (\cite[Proposition 3.2]{EKQR}).

    \begin{rmk}\label{Appendix.B.Cc.Density}\rm
        The algebraic tensor product $X\odot C_c(G)$ is dense in $X\rtimes_{\gamma,r}G$. This is because $X\odot C_c(G)$ is $L^1$-norm dense in $C_c(G,X)$ and the crossed product norm on $C_c(G,A)$ is dominated by its $L^1$-norm.
    \end{rmk}

    The proof of the next theorem is only sketched.

    \begin{thm}\label{Appendix.B.Thm.}
        Let $(\gamma,\alpha)$ be an action of a locally compact group $G$ on a $C^*$-cor\-re\-spond\-ence $(X,A)$. Then, there exists an injective correspondence homomorphism
            \[(\psi_\gamma,\pi_\alpha):(X\rtimes_{\gamma,r}G,A\rtimes_{\alpha,r}G)
            \rightarrow(\mathcal{L}_A(A\otimes L^2(G),X\otimes L^2(G)),\mathcal{L}_A(A\otimes L^2(G)))\]
        such that
        \begin{align*}\label{Appendix.B.Int.Form}
            \big(\psi_\gamma(x)h\big)(r) &=\int_G\gamma_{r}^{-1}(x(s))\cdot h(s^{-1}r)\,ds, \\ 
            \big(\pi_\alpha(f)h\big)(r) &=\int_G\alpha_r^{-1}(f(s))\,h(s^{-1}r)\,ds
        \end{align*}
        for $x\in C_c(G,X)$, $f\in C_c(G,A)$, $h\in C_c(G,A)$, and $r\in G$.
    \end{thm}

    \begin{proof}
        It is well-known that $\pi_\alpha$ gives a nondegenerate embedding.

        For each $x\in C_c(G,X)\subseteq X\rtimes_{\gamma,r}G$, define $\rho_\gamma(x):C_c(G,X)\rightarrow C_c(G,A)$ by
            \[(\rho_\gamma(x)k)(r)=\int_G\Delta(r^{-1})\langle\gamma_{s}^{-1}(x(sr^{-1})),k(s)\rangle_A\,ds\]
        for $k\in C_c(G,X)$ and $r\in G$, where $\Delta$ is the modular function of $G$. A routine computation yields
        \begin{equation}\label{Appendix.B.Adj.Formula}
            \langle\psi_\gamma(x)h,k\rangle_A=\langle h,\rho_\gamma(x)k\rangle_A,\quad
            \rho_\gamma(x)\psi_\gamma(y)=\pi_\alpha(\langle x,y\rangle_{A\rtimes_{\alpha,r}G})
        \end{equation}
        for $h\in C_c(G,A)\subseteq A\otimes L^2(G)$, $k\in C_c(G,X)\subseteq X\otimes L^2(G)$, and $x,y\in C_c(G,X)$. This shows that $\psi_\gamma$ and $\rho_\gamma$ both extend continuously to all of $X\rtimes_{\gamma,r}G$, and $\psi_\gamma(x)\in\mathcal{L}_A(A\otimes L^2(G),X\otimes L^2(G))$ for $x\in X\rtimes_{\gamma,r}G$ with the adjoint $\psi_\gamma(x)^*=\rho_\gamma(x)$.

        The second relation in \eqref{Appendix.B.Adj.Formula} gives 
            $\langle\psi_\gamma(x),\psi_\gamma(y)\rangle_{\mathcal{L}_A(A\otimes L^2(G))}=\pi_\alpha(\langle x,y\rangle_{A\rtimes_{\alpha,r}G})$
        for $x,y\in X\rtimes_{\gamma,r}G$. Let $a\in A$ and $\phi\in C_c(G)\subseteq C^*_r(G)$. Since the strict extension $\overline{\pi_\alpha}$ embeds $A$ into $\mathcal{L}_A(A\otimes L^2(G))$ such that $\big(\overline{\pi_\alpha}(a)h\big)(r)=\alpha_r^{-1}(a)h(r)$, we deduce that
            \begin{equation*}\label{Appendix.B.Cov.Rep.1}
                \big(\varphi_{\mathcal{L}_A(A\otimes L^2(G))}(\overline{\pi_\alpha}(a))\,k\big)(r)=\varphi_A(\alpha_r^{-1}(a))k(r)
            \end{equation*}
        for $k\in C_c(G,X)$ and $r\in G$. Similarly,
            \begin{equation*}\label{Appendix.B.Cov.Rep.2}
                \big(\varphi_{\mathcal{L}_A(A\otimes L^2(G))}\big(\overline{\pi_\alpha}(\phi)\big)k\big)(r)=\int_G\phi(s)k(s^{-1}r)\,ds
            \end{equation*}
        for $k\in C_c(G,X)$ and $r\in G$. An easy computation then verifies
            \[\big(\psi_\gamma\big(\varphi_{A\rtimes_{\alpha,r}G}(a\otimes\phi)x\big)h\big)(r)=\big(\varphi_{\mathcal{L}_A(A\otimes L^2(G))}\big(\pi_\alpha(a\otimes\phi)\big)\psi_\gamma(x)h\big)(r)\]
        for $x\in C_c(G,X)$, $h\in C_c(G,A)$, and $r\in G$, which gives
            \begin{equation*}\label{Appendix.B.Eqn.7}
                \psi_\gamma\big(\varphi_{A\rtimes_{\alpha,r}G}(a\otimes\phi)\,x\big)=\varphi_{\mathcal{L}_A(A\otimes L^2(G))}\big(\pi_\alpha(a\otimes\phi)\big)\psi_\gamma(x).
            \end{equation*}
        The same equality is now true for $f\in A\rtimes_{\alpha,r}G$ in place of $a\otimes\phi$ and for $x\in X\rtimes_{\gamma,r}G$, and hence $(\psi_\gamma,\pi_\alpha)$ is a correspondence homomorphism. Finally, it is injective since $\pi_\alpha$ is injective.
    \end{proof}

    Let $(\gamma,\alpha)$ be an action of $G$ on $(X,A)$, and $(\sigma^\gamma,\delta^\alpha)$ be the corresponding coaction. Define
        \begin{equation}\label{Appendix.B.Cor.1.Eqn}
            \sigma^\gamma_G=\overline{{\rm id}_X\otimes\check{\mu}_G}\circ\sigma^\gamma,\quad
            \delta^\alpha_G=\overline{{\rm id}_X\otimes\check{\mu}_G}\circ\delta^\alpha,
        \end{equation}
    where $\check{\mu}_G:C_0(G)\rightarrow S_{\widehat{W}_G}$ is the Hopf $C^*$-algebra isomorphism given in \eqref{Prel.pi.and.U}. Then $(\sigma^\gamma_G,\delta^\alpha_G)$ is a coaction of $S_{\widehat{W}_G}$ on $(X,A)$. In the next corollary, we regard $\sigma^\gamma_{G\iota}(X)=\overline{{\rm id}_X\otimes\iota_{S_{\widehat{W}_G}}}(\sigma^\gamma_G(X))$ as a subspace of $\mathcal{L}_A(A\otimes L^2(G),X\otimes L^2(G))$.

    \begin{cor}\label{Appendix.B.Cor.1}
        Let $(\gamma,\alpha)$ be an action of a locally compact group $G$ on $(X,A)$. Then $(\psi_\gamma,\pi_\alpha)$ in Theorem \ref{Appendix.B.Thm.} gives an isomorphism from $(X\rtimes_{\gamma,r}G,A\rtimes_{\alpha,r}G)$ onto $(X\rtimes_{\sigma^\gamma_G}\widehat{S}_{\widehat{W}_G},A\rtimes_{\delta^\alpha_G}\widehat{S}_{\widehat{W}_G})$ such that
            \begin{equation}\label{Appendix.B.Cor.1.F}
                \psi_\gamma(\xi\otimes\phi)=\sigma^\gamma_{G\iota}(\xi)\cdot(1_{M(A)}\otimes\phi),\quad
                \pi_\alpha(a\otimes\phi)=\delta^\alpha_{G\iota}(a)(1_{M(A)}\otimes\phi)
            \end{equation}
        for $\xi\in X$, $a\in A$, and $\phi\in C_c(G)$.
    \end{cor}

    \begin{proof}
        We only need to prove that $\psi_\gamma$ satisfies the first equality in \eqref{Appendix.B.Cor.1.F} and gives a surjection onto $X\rtimes_{\sigma^\gamma_G}\widehat{S}_{\widehat{W}_G}$. Let $\xi\in X$ and $\phi\in C_c(G)$. We see from Remark \ref{Appendix.B.Rmk.EmbedA} that
            \[(\sigma^\gamma_{G\iota}(\xi)\,h)(r)=\gamma_r^{-1}(\xi)\cdot h(r)\]
        for $h\in C_c(G,A)$ and $r\in G$. Hence
            \[\big(\psi_\gamma(\xi\otimes\phi)h\big)(r) = \gamma_r^{-1}(\xi)\cdot\int_G\phi(s)h(s^{-1}r)\,ds =
            \big(\sigma^\gamma_{G\iota}(\xi)\big((1_{M(A)}\otimes\phi)h\big)\big)(r),\]
        which shows the first equality in \eqref{Appendix.B.Cor.1.F}. Since $X\odot C_c(G)$ is dense in $X\rtimes_{\gamma,r}G$ by Remark \ref{Appendix.B.Cc.Density} and $\psi_\gamma$ is isometric, we must have $\psi_\gamma(X\rtimes_{\gamma,r}G)=X\rtimes_{\sigma^\gamma_G}\widehat{S}_{\widehat{W}_G}$.
    \end{proof}

    We now provide a proof of Corollary \ref{Cor1.to.Main.Thm}.

    \begin{proof}[Proof of Corollary \ref{Cor1.to.Main.Thm}]
        Let
            \[\zeta_G=\overline{{\rm id}_{\mathcal{O}_X}\otimes\check{\mu}_G}\circ\zeta.\]
        Clearly, $\zeta_G$ is the coaction of $S_{\widehat{W}_G}$ on $\mathcal{O}_X$ induced by $(\sigma^\gamma_G,\delta^\alpha_G)$. Define a representation
            \[(k_X\rtimes_\gamma G,k_A\rtimes_\gamma G):(X\rtimes_{\gamma,r}G,A\rtimes_{\alpha,r}G)\rightarrow\mathcal{O}_X\rtimes_{\beta^\zeta,r}G\]
        to be the composition as indicated in the following diagram:
            \[ \xymatrix{(X\rtimes_{\gamma,r}G,A\rtimes_{\alpha,r}G) \ar[rr]^-{(\psi_\gamma,\pi_\alpha)} \ar @{-->}[d]_-{(k_X\rtimes_\gamma G,k_A\rtimes_\alpha G)}
            && (X\rtimes_{\sigma^\gamma_G}\widehat{S}_{\widehat{W}_G},A\rtimes_{\delta^\alpha_G}\widehat{S}_{\widehat{W}_G})
            \ar[d]^-{(k_X\rtimes{\rm id}_{\widehat{S}_{\widehat{W}_G}},k_A\rtimes{\rm id}_{\widehat{S}_{\widehat{W}_G}})} \\
                \mathcal{O}_X\rtimes_{\beta^\zeta,r}G && \mathcal{O}_X\rtimes_{\zeta_G}\widehat{S}_{\widehat{W}_G}. \ar @{=}[ll] } \]
        By definition (\eqref{Sec.5.Rep.1} and \eqref{Appendix.B.Cor.1.F}), we have
            \[(k_X\rtimes_\gamma G)(f)(r)=k_X(f(r)),\quad(k_A\rtimes_\alpha G)(g)(r)=k_A(g(r))\]
        for $f\in C_c(G,X)$, $g\in C_c(G,A)$, and $r\in G$. The conclusion then follows by Theorem \ref{Main.Theorem.}.
    \end{proof}

}



\end{document}